\documentclass[a4paper]{amsart}
\usepackage{latexsym,bm}
\usepackage{amsmath,amsfonts,amssymb,mathrsfs}
\usepackage{cite}

\newcommand{\BS}{\mathfrak S}
\let\Sym=\BS
\newcommand{\Z}{\mathbb Z}
\newcommand{\HH}{\mathscr{H}}
\let\H=\HH

\newcommand{\calD}{{\mathcal D}}

\renewcommand\O{{\mathcal{O}}}

\newcommand{\eps}{\varepsilon}

\newcommand\blam{{\boldsymbol\lambda}}
\newcommand\bmu{{\boldsymbol\mu}}
\DeclareMathOperator\comp{comp}
\DeclareMathOperator\Shape{Shape}
\DeclareMathOperator\Std{Std}
\newcommand\s{\mathfrak{s}}
\renewcommand\t{\mathfrak{t}}
\renewcommand\u{\mathfrak{u}}
\renewcommand\v{\mathfrak{v}}

\let\gedom=\trianglerighteq
\let\gdom=\vartriangleright

\DeclareMathOperator{\ind}{\!\uparrow}
\DeclareMathOperator{\res}{\!\downarrow}

\newcommand{\bQ}{\mathbf{Q}}
\renewcommand{\b}{\mathbf{b}}
\renewcommand{\c}{\mathbf{c}}
\newcommand{\bt}{\mathbf{t}}
\DeclareMathOperator{\Hom}{Hom}
\DeclareMathOperator{\Irr}{Irr}
\DeclareMathOperator{\End}{End}

\newcounter{main}
\theoremstyle{plain}
\newtheorem*{Notation}{Notation}
\newtheorem{Theorem}[main]{Theorem}
\swapnumbers
\numberwithin{equation}{section}
\newtheorem{prop}[equation]{Proposition}
\newtheorem{thm}[equation]{Theorem}
\newtheorem{cor}[equation]{Corollary}
\newtheorem{lem}[equation]{Lemma}
\newtheorem{dfn}[equation]{Definition}
\theoremstyle{remark}
\newtheorem{Remark}[equation]{Remark}

\def\map#1#2{\,{:}\,#1\!\longrightarrow\!#2}
\def\And{\text{ and }}
{\catcode`\|=\active
  \gdef\set#1{\mathinner{\lbrace\,{\mathcode`\|"8000%
                                   \let|\midvert #1}\,\rbrace}}
}
\def\midvert{\egroup\mid\bgroup}

\title[Morita equivalences of cyclotomic Hecke algebras]%
   {Morita equivalences of cyclotomic Hecke algebras of type $G(r,p,n)$}
\author{Jun Hu}
\address{Department of Applied Mathematics,
Beijing Institute of Technology,
Beijing, 100081, P.R. China}
\email{junhu303@yahoo.com.cn}
\author{Andrew Mathas}
\address{School of Mathematics and Statistics F07,
University of Sydney, NSW 2006, Australia}
\email{a.mathas@usyd.edu.au}

\begin{document}
\bibliographystyle{andrew}

\begin{abstract}
We prove a Morita reduction theorem for the cyclotomic Hecke
algebras $\HH_{r,p,n}({q,\bQ})$ of type $G(r,p,n)$ with $p>1$ and
$n\geq 3$. As a consequence, we show that computing the
decomposition numbers of $\HH_{r,p,n}(\bQ)$ reduces to computing the
\textit{$p'$-splittable decomposition numbers} (see
Definition~\ref{split}) of the cyclotomic Hecke algebras
$\HH_{r',p',n'}(\bQ')$, where $1\leq r'\leq r$, $1\leq n'\leq n$, $
p'\mid p$ and where the parameters~$\bQ'$ are contained in a single
$(\varepsilon',q)$-orbit and $\varepsilon'$ is a primitive $p'$th root of
unity.
\end{abstract}

\maketitle

\begin{center}\it
  Dedicated to Toshiaki Shoji on the occasion of his sixtieth birthday.
\end{center}

\section{Introduction}
Motivated by ``generic features'' of the representation theory of
finite reductive groups Brou\'e and Malle~\cite{BM:cyc} attached a
\textit{cyclotomic Hecke algebra} to each complex reflection group. These
algebras have many good properties and, conjecturally, they arise as
the endomorphism algebras of Deligne--Lusztig representations.

This paper is concerned with the cyclotomic Hecke algebras of type
$G(r,p,n)$ with $p>1$ and $n\geq 3$. These algebras were first
considered by Brou\'e and Malle~\cite{BM:cyc} and by
Ariki~\cite{Ariki:hecke} in the semisimple case. These algebras have
been studied extensively in the non--semisimple case, notably by the
first author~\cite{Hu:ModGppn,Hu:simpleGrpn,Hu:CrystalGppn,Hu:Grpn}
and by Genet and Jacon~\cite{GenetJacon}. In particular, the simple
modules of these algebras have been classified over any field of
characteristic coprime to $p$~\cite{Hu:simpleGrpn}.

In the case $p=1$ the cyclotomic Hecke algebras of type $G(r,1,n)$ are
known as the Ariki--Koike algebras. These algebras are well
understood; see ~\cite{m:cyclosurv} and the references therein. The
highlight of this theory is Ariki's celebrated theorem which says that
the decomposition numbers of these algebras in characteristic zero can
be computed using the canonical bases of the higher level Fock spaces
for the quantized affine special linear groups. Another fundamental result for
the Ariki--Koike algebras is the Morita equivalence theorem of Dipper
and the second author~\cite{DM:Morita} which says that, up to Morita
equivalence, these algebras are determined by the $q$--orbits of their
parameters.

The first main result in this paper gives an analogue of the
Dipper--Mathas Morita equivalence theorem for the Hecke algebras of
type $G(r,p,n)$. To state this result explicitly, fix positive
integers $r$, $p$ and $n$ with $r=pt$, for some integer~$t$, and let
$K$ be an algebraically closed field of characteristic coprime to
$p$. Fix parameters $q,Q_1,\dots,Q_t\in K^\times$ and let
$\bQ=(Q_1,\dots,Q_t)$. Let $\HH_{r,n}(\bQ)$ be the Ariki--Koike
algebra and let $\HH_{r,p,n}(\bQ)$ be the Hecke algebra of type
$G(r,p,n)$ with parameters~$q$ and $\bQ$. The algebra
$\HH_{r,n}(\bQ)$ is equipped with an automorphism $\sigma$ of order
$p$ and $\HH_{r,p,n}(\bQ)$ is the fixed point subalgebra of
$\HH_{r,n}(\bQ)$ under $\sigma$. There is a second automorphism
$\tau$ on $\HH_{r,n}$ which fixes $\HH_{r,p,n}$ setwise. For the
precise definitions see the third paragraph in Section 2,
Definition~\ref{dfn21} and Definition~\ref{sigma tau}.

Fix a primitive $p$th root of unity $\eps$ in $K$ and say that $Q_i$
and $Q_j$ are in the \textbf{same $(\eps,q)$--orbit} if $Q_i=\eps^a
q^bQ_j$, for some integers $a,b\in\Z$. Given two ordered tuples
$\mathbf X=(X_1,\dots,X_s)$ and $\mathbf Y=(Y_1,\dots,Y_l)$ of
elements of $K^\times$ set $\mathbf X\vee\mathbf
Y=(X_1,\dots,X_s,Y_1,\dots,Y_l)$.

Suppose that $A$ is an algebra and that $\Z_p$ is a group which acts
on~$A$ as a group of algebra automorphisms. Let $A\rtimes \Z_p$ be
the algebra with elements $\set{ag|a\in A\text{ and }g\in \Z_p}$ and
with multiplication $$ag\cdot bh=ab^g\cdot gh,\qquad \text{for
}a,b\in A\text{ and }g,h\in \Z_p.$$

The first main result of this paper is the following.

\begin{Theorem}\label{main1}
Suppose that $\bQ=\bQ_1\vee\dots\vee\bQ_{\kappa}$, where
$Q_i\in\bQ_{\alpha}$ and $Q_j\in\bQ_{\beta}$ are in the same
$(\eps,q)$--orbit only if $\alpha=\beta$. Let
$t_\alpha=|\bQ_{\alpha}|$, for $1\le\alpha\le\kappa$. Then
$\HH_{r,p,n}(\bQ)$ is Morita equivalent to the algebra
$$\displaystyle\bigoplus_{\substack{b_1,\cdots,b_{\kappa}\geq 0\\
b_1+\cdots+b_{\kappa}=n}} \Big(\HH_{pt_1,b_1}(\bQ_1)\otimes\cdots
\otimes\HH_{pt_{\kappa},b_{\kappa}}(\bQ_{\kappa})\Big)\rtimes\Z_p.$$
\end{Theorem}

In the theorem, each of the algebras
$\HH_{pt_\alpha,b_\alpha}(\bQ_{\alpha})$ has an automorphism
$\sigma_\alpha$ of order~$p$ and, in the direct sum, the
automorphism $\sigma_1\otimes\dots\otimes\sigma_\kappa$ acts
diagonally on the algebra $\HH_{pt_1,b_1}(\bQ_1)\otimes\cdots
\otimes\HH_{pt_{\kappa},b_{\kappa}}(\bQ_{\kappa})$. Observe that
$\langle\sigma_1\otimes\dots\otimes\sigma_\kappa\rangle\cong\Z_p$.

The second result of this paper uses Theorem~\ref{main1} to prove a
reduction theorem for computing the decomposition numbers of
$\HH_{r,p,n}(\bQ)$. In order to state this result fix a modular
system $(F,\O,K)$ ``with parameters''. That is, we fix an
algebraically closed field~$F$ of characteristic zero, a discrete
valuation ring $\O$ with maximal ideal $\pi$ and residue field
$K\cong\O/\pi$, together with parameters
$\hat{q},\hat{Q}_1,\dots,\hat{Q}_t\in\O^\times$ such that
$q=\hat{q}+\pi$ and $Q_i=\hat{Q}_i+\pi$ for each $i$. Let
$\HH^F_{r,p,n}=\HH^F_{r,p,n}(\hat{\bQ})$ be the Hecke algebra of
type $G(r,p,n)$ over $F$ with parameters $\hat{q}$ and
$\hat{\bQ}=(\hat{Q}_1,\dots,\hat{Q}_t)$ and similarly let
$\HH_{r,p,n}^\O=\HH_{r,p,n}(\hat{\bQ})$ and write
$\HH_{r,p,n}^K=\HH_{r,p,n}(\bQ)$. We assume that $\HH_{r,p,n}^F$ is
semisimple. By freeness we have that
$\HH^F_{r,p,n}\cong\HH^\O_{r,p,n}\otimes_\O F$ and
$\HH^K_{r,p,n}\cong\HH^\O_{r,p,n}\otimes_\O K$. Thus, by choosing
$\O$--lattices we can talk of modular reduction from
$\HH_{r,p,n}^F$--Mod to $\HH_{r,p,n}^K$--Mod.

Using the definitions, it is straightforward to check that the
automorphisms $\sigma$ and $\tau$ commute with modular reduction.
Thus, we have compatible automorphisms $\sigma$ and $\tau$ on
$\H_{r,n}^F$ and on $\HH_{r,n}^K$.

Let $R\in\{F,K\}$ and let $M$ be an $\HH^R_{r,p,n}$--module.  Then we
define a new $\HH^R_{r,p,n}$--module $M^\tau$ by ``twisting'' the
action of $\HH^R_{r,p,n}$using the automorphism $\tau$.
Explicitly, $M^\tau=M$ as a vector space and the
$\HH^R_{r,p,n}$--action on $M^\tau$ is defined by
$$m\cdot h=m\tau(h),\qquad
                \text{for all }m\in M\text{ and }h\in\HH^R_{r,p,n}.$$
Since $\tau^{p}$ is an inner automorphism of the algebra
$\HH^R_{r,p,n}$, it follows that $M\cong M^{\tau^p}$ for any
$\HH^R_{r,p,n}$--module $M$. Therefore, there is a natural action of
the cyclic group $\Z_p$ on the set of isomorphism classes of
$\HH^R_{r,p,n}$--modules. We define the {\bf inertia group} of $M$
to be $G_M=\set{k|0\leq k<p, M\cong M^{\tau^k}}\le\Z_p$.

If $A$ is any algebra let $\Irr(A)$ be the complete set of
isomorphism classes of irreducible $A$--modules. We are interested
in the inertia group $G_S$ and $G_D$, for $S\in\Irr(\HH^F_{r,p,n})$
and $D\in\Irr(\HH^K_{r,p,n})$.

\begin{dfn}\label{split}
Suppose that $S\in\Irr(\HH^F_{r,p,n})$ and
$D\in\Irr(\HH^K_{r,p,n})$. The decomposition number $[S:D]$ is a
{\bf $p$-splittable} decomposition number of $\HH_{r,p,n}(\bQ)$ if
$G_{S}=\{0\}=G_{D}$.
\end{dfn}

The second main result of this paper is the following:

\begin{Theorem} \label{main2} Then the decomposition
    numbers of the cyclotomic Hecke algebras of type $G(r,p,n)$ are
    completely determined by the $p'$-splittable decomposition
    numbers of certain cyclotomic Hecke algebras
    $\HH_{r',p',n'}(\bQ')$, where~$p'$ divides $p$, $1\le r'\le r$,
    $1\le n'\le n$ and where the parameters $\bQ'$ are contained in a single
    $(\varepsilon',q)$--orbit and $\varepsilon'$ is a primitive $p'$th root of unity.
\end{Theorem}

The proof of Theorem~\ref{main2} explicitly describes the algebras
$\HH_{r',p',n'}(\bQ')$ and the parameters $\bQ'$ which appear in
this reduction. Thus, once the $p'$-splittable decomposition numbers
are known this result gives an algorithm for computing the
decomposition matrices of the cyclotomic Hecke algebras of type
$G(r,p,n)$.

This paper is organized as follows. In the next section we define
the cyclotomic Hecke algebras of type $G(r,p,n)$ and prove the
Morita equivalence result for the Hecke algebras of type $G(r,1,n)$
which underpins all of the results in this paper. In the third
section we apply the results for the algebras of type $G(r,1,n)$ to
prove Theorem~\ref{main1}. The fourth section of the paper uses
Clifford theory to show that if an algebra can be written as a
semidirect product then its decomposition numbers are determined by
a suitable family of $p'$-splittable decomposition numbers. This
result is then applied in section~5 to prove Theorem~\ref{main2}.

\section{Morita equivalence theorems for Hecke algebras of type $G(r,1,n)$}
In this section we define the cyclotomic Hecke algebras and set our
notation. We then recall and generalize the Morita equivalence results
that we need for the cyclotomic algebras of type $G(r,1,n)$.

Throughout this paper we fix positive integers $r$, $p$ and $n$ such
that $r=pt$ for some integer~$t$. Let $K$ be an algebraically closed
field which contains a primitive $p$th root of unity~$\eps$. In
particular, the characteristic of $K$ is coprime to $p$. {\it
Throughout this paper, we assume that $p>1$ and $n\geq 3$.} Fix
parameters $q,Q_1,\cdots,Q_t\in K^{\times}$ and, as in the
introduction, let $\bQ:=(Q_1,\cdots,Q_t)$ and write $|\bQ|=t$.

Let $\HH_{r,n}(\bQ)$ be the cyclotomic Hecke algebra of type
$G(r,1,n)$. As a $K$-algebra $\HH_{r,n}(\bQ)$ is generated by
$T_0,T_1,\cdots,T_{n-1}$ subject to
the relations:
$$\begin{aligned}
& (T_{0}^{p}-Q_1^p)
(T_0^p-Q_2^p)\cdots (T_0^p-Q_t^p)=0,\\
& T_0T_1T_0T_1=T_1T_0T_1T_0,\\
& (T_i+1)(T_i-q)=0,\quad 1\leq i\leq n-1,\\
& T_iT_{i+1}T_i=T_{i+1}T_iT_{i+1},\quad 1\leq i\leq n-2,\\
& T_iT_j=T_jT_i,\quad 0\leq i<j-1\leq n-2.
\end{aligned}
$$
That is, $\HH_{r,n}(\bQ)$ is the cyclotomic Hecke algebra of type
$G(r,1,n)$ with parameters $\{Q_1',\dots,Q_r'\}$, where
$Q_{i+pj}'=\eps^i Q_{j+1}$ for $0\le i<p$ and $0\le j<t$.

\begin{dfn} \label{dfn21}
The \textbf{cyclotomic Hecke algebra of type $G(r,p,n)$} is the
subalgebra $\HH_{r,p,n}(\bQ)$ of $\HH_{r,n}(\bQ)$ which is generated
by the elements $T_0^p,T_u=T_0^{-1}T_1T_0$ and
$T_1,T_2,\cdots,T_{n-1}$.
\end{dfn}

When the choice of parameters $q,\bQ$ is clear we write
$\HH_{r,p,n}=\HH_{r,p,n}(\bQ)$ and $\HH_{r,n}=\HH_{r,n}(\bQ)$.
When we want to emphasize the coefficient ring we write
$\HH_{r,p,n}^K=\HH^K_{r,p,n}(\bQ)$ and
$\HH_{r,n}^K=\HH^K_{r,n}(\bQ)$, respectively.

\begin{Remark}
  As noted by Malle~\cite[\S4.B]{Malle:fake}, if $n=2$ then the Hecke
  algebra of type $G(2m,2p,2)$ cannot be identified with a subalgebra of
  the Hecke algebra of type $G(2m,1,2)$. It is for this reason that we
  assume that $n\ge3$ in this paper (all of the results in this section are
  valid for $n\ge1$.
\end{Remark}

Let $\Sym_n$ be the symmetric group on $n$ letters. As the type $A$
braid relations hold in $\HH_{r,n}$ for each $w\in\Sym_n$ there is a
well--defined element $T_w\in\HH_{r,n}$, where $T_w=T_{i_1}\dots
T_{i_k}$ whenever $k$ is minimal such that
$w=(i_1,i_1+1)\dots(i_k,i_k+1)$. Set $L_1=T_0$ and
$L_{k+1}=q^{-1}T_kL_kT_k$, for $k=1,\dots,n-1$. Then Ariki and
Koike~\cite[Theorem~3.10]{AK} showed that
$$\Big\{L_1^{c_1}\dots L_n^{c_n}T_w\,\Big|\,%
   w\in\Sym_n\text{ and }0\le c_i<r\Big\}$$
is a basis of $\HH_{r,n}$.

The Morita equivalence in Theorem~\ref{main1} is a consequence of
the following result for the cyclotomic Hecke algebras~$\H_{r,n}$.

\begin{thm}[\protect{Dipper--Mathas~\cite[Theorem 1.1]{DM:Morita}}]
\label{DM:Morita} Suppose that $\bQ=\bQ_1\vee\dots\vee\bQ_{\kappa}$,
such that $\alpha=\beta$ whenever $Q_i\in\bQ_{\alpha}$,
$Q_j\in\bQ_{\beta}$ and $Q_i=q^a\eps^bQ_j$, for some $a,b\in\Z$. Let
$t_\alpha=|\bQ_{\alpha}|$. Then $\HH_{r,n}(\bQ)$ is Morita
equivalent to the algebra
$$\bigoplus_{\substack{b_1,\dots,b_\kappa\ge0\\b_1+\dots+b_\kappa=n}}
\HH_{pt_1,b_1}(\bQ_1)\otimes\dots\otimes
           \HH_{pt_\kappa,b_\kappa}(\bQ_{\kappa}).$$
\end{thm}

As noted in \cite{DM:Morita} the proof of Theorem~\ref{DM:Morita}
quickly reduces to the case $\kappa=2$, so only this case is
considered in~\cite{DM:Morita}. Unfortunately, to prove
Theorem~\ref{main1} we need detailed information about the bimodule
which induces the Morita equivalence of Theorem~\ref{DM:Morita} for
arbitrary $\kappa\ge1$. Consequently, we need to generalize the
results of~\cite{DM:Morita} and construct the bimodule which induces
the Morita equivalence of Theorem~\ref{DM:Morita} (in the special
case when  $\bQ$ is partitioned into a disjoint union of
$(\eps,q)$--orbits). In constructing this bimodule we refer the
reader back to \cite{DM:Morita} whenever the details are not
substantially different from the case $\kappa=2$.

First, fix non-negative integers $a$ and $b$ with $a+b\le n$ and
an integer $s$ with $1\le s\le t$. Define
$$v_{a,b}(s)=\prod_{k=1}^s(L_1^p-Q_k^p)\dots(L_a^p-Q_k^p) \cdot
  T_{w_{a,b}}\cdot\prod_{k=s+1}^t(L_1^p-Q_k^p)\dots(L_b^p-Q_k^p),$$
where $w_{a,b}=(s_{a+b-1}\dots s_1)^b$. (So $v_{n-b,b}(s)$ is the
element $v_b$ of \cite[Definition~3.3]{DM:Morita}.) We write $
v_{a,b}^{+}(s)=\prod_{k=1}^s(L_1^p-Q_k^p)\dots(L_a^p-Q_k^p) \cdot
  T_{w_{a,b}}$.
It may help the
reader to observe that if we write $w_{a,b}\in\BS_{a+b}$ as a
permutation in two-line notation then
$$w_{a,b}=\Big(\begin{array}{*6c}
1&\cdots&a&a+1&\cdots&a+b\\b+1&\cdots&a+b&1&\cdots&b
\end{array}\Big).$$

We will use the following notation extensively.\medskip

\begin{Notation}
Given any sequence $\mathbf a=(a_1,\dots,a_k)$ and integers $1\le
i\le j\le k$ we set $a_{i..j}=a_i+\dots+a_j$. If $i<j$ then set
$a_{j..i}=0$.
\end{Notation}

Until further notice we fix a partition
$\bQ=\bQ_1\vee\dots\vee\bQ_\kappa$ of $\bQ$ such that
$Q_i\in\bQ_\alpha$ and $Q_j\in\bQ_\beta$ are in the same
$(\eps,q)$-orbit only if $\alpha=\beta$. Set
$\bt=(t_1,\dots,t_\kappa)$ where $t_\alpha=|\bQ_\alpha|$, for $1\le
\alpha\le\kappa$. Without loss of generality we assume that
$\bQ_\alpha=(Q_{t_{1..\alpha-1}+1},\dots,Q_{t_{1..\alpha}})$, for
$\alpha=1,\dots,\kappa$ (set $t_0=0$).

Let $\Lambda(n,\kappa)=\set{\b=(b_1,\dots,b_\kappa)|b_{1..\kappa}=n
\text{ and } b_\alpha\ge0 \text{ for }1\le \alpha\le\kappa}$ be the
set of compositions of $n$ into $\kappa$ parts. If
$\b\in\Lambda(n,\kappa)$ then, for convenience, we set
$b_{\kappa+1}=0$.

\begin{dfn}Suppose that $\b\in\Lambda(n,\kappa)$.  Define
$$v_\b=
v^{+}_{b_\kappa,b_{1..\kappa-1}}(t_{1..\kappa-1})
  v^{+}_{b_{\kappa-1},b_{1..\kappa-2}}(t_{1..\kappa-2})
 \dots v^{+}_{b_2,b_{1..1}}(t_{1..1})u_{\omega_{\b}}^{+}, $$
where $$u_{\omega_{\b}}^{+}:=\biggl(\prod_{k=t_1+1}^{t_{1..2}}(L_1^p-Q_k^p)\dots(L_{b_1}^p-Q_k^p)\biggr)\dots
\biggl(\prod_{k=t_{1..\kappa-1}+1}^{t_{1..\kappa}}(L_1^p-Q_k^p)\dots(L_{b_{1..\kappa-1}}^p-Q_k^p)\biggr),
$$
and let $w_\b=w_{b_\kappa,b_{1..\kappa-1}}
w_{b_{\kappa-1},b_{1..\kappa-2}} \dots w_{b_2,b_{1..1}}$. Define
$V^\b=v_\b\HH_{r,n}$.
\end{dfn}

Note that $v_\b$ depends crucially on our fixed partition
$\bQ=\bQ_1\vee\dots\vee\bQ_\kappa$ of $\bQ$, so we should really write
$v_\b=v_\b(\bQ_1,\dots,\bQ_\kappa)$. Our first goal is to understand~$V^\b$.
Note that for each $2\leq\alpha\leq\kappa$, $u_{\omega_{\b}}^{+}$ has a factor
$\prod_{k=t_{1..\alpha-1}+1}^{t}(L_1^p-Q_k^p)\dots(L_{b_{1..\alpha-1}}^p-Q_k^p)$, i.e., $$
u_{\omega_{\b}}^{+}=\prod_{k=t_{1..\alpha-1}+1}^{t}(L_1^p-Q_k^p)\dots(L_{b_{1..\alpha-1}}^p-Q_k^p)B(\alpha),
$$
for some polynomial $B(\alpha)$ in the Murphy operators. Then $$\begin{aligned}
v_{\b}&=v^{+}_{b_\kappa,b_{1..\kappa-1}}(t_{1..\kappa-1})\dots v^{+}_{b_{\alpha+1},b_{1..\alpha}}(t_{1..\alpha})
v_{b_{\alpha},b_{1..\alpha-1}}(t_{1..\alpha-1})\\
&\quad v^{+}_{b_{\alpha-1},b_{1..\alpha-2}}(t_{1..\alpha-2})\dots
v^{+}_{b_{2},b_{1..1}}(t_{1..1})B(\alpha).\end{aligned}
$$
This observation will be used in the proof of the following two key properties of the elements $v_\b$.

\begin{prop}\label{v-shift}
Suppose that $\b\in\Lambda(n,\kappa)$. Then
\begin{enumerate}
\item $T_iv_\b=v_\b T_{w_\b(i)}$, whenever $1\le i<n$ and
        $i\ne b_{\alpha..\kappa}$ for $\alpha=1,\dots,\kappa$.
\item $L_kv_\b=v_\b L_{w_\b(k)}$, whenever $1\le k\le n$.
\end{enumerate}
\end{prop}

\begin{proof}
After translating notation, \cite[Prop.~3.4]{DM:Morita} says that
if $1\le s\le t$ and $a$ and $b$ are non--negative integers with
$a+b\le n$ then
\begin{equation}\label{switch}
T_iv_{a,b}(s)=v_{a,b}(s)T_{w_{a,b}(i)}\quad\text{and}\quad
  L_k v_{a,b}(s)=v_{a,b}(s)L_{w_{a,b}(k)}
\end{equation}
whenever $1\le i<a+b$, $i\ne a$, and $1\le k\le a+b$. This is
precisely the special case of the Proposition when $\kappa=2$. The
general case follows from this result, the observation above Proposition \ref{v-shift} and the fact
that $T_iv_{a,b}(s)=v_{a,b}(s)T_i$ and $L_kv_{a,b}(s)=v_{a,b}(s)L_k$
whenever $a+b<i<n$ and $a+b<k\le n$ for non--negative integers $a$
and~$b$.
\end{proof}

Observe that $v_\b T_j=T_{w_\b^{-1}(j)}v_\b$ and $v_\b
L_m=L_{w_\b^{-1}(m)}v_\b$ by Proposition~\ref{v-shift}, for $1\le
j<n$, $1\le m\le n$ with $j\ne b_{1..\alpha}$ for
$\alpha=1,\dots,\kappa$.

\begin{lem}\label{v-vanishing}
Suppose that $\b\in\Lambda(n,\kappa)$, $1\le\alpha\le\kappa$ and
$b_{\alpha}\neq 0$. Then
$$\prod_{Q_i\in\bQ_\alpha}(L_{1+b_{\alpha+1..\kappa}}^p-Q_i^p)
     \cdot v_\b=0
 =v_\b\cdot\prod_{Q_i\in\bQ_\alpha}(L_{b_{1..\alpha-1}+1}^p-Q_i^p)
$$
\end{lem}

\begin{proof}
Recall that $\prod_{i=1}^t(L_1^p-Q_i^p)=0$ since $L_1=T_0$.
Therefore, it follows from the definitions that if $b_{\kappa}\neq 0$ then
$$\prod_{Q_i\in\bQ_\kappa}(L_1^p-Q_i^p)\cdot
  v_{b_\kappa,b_{1..\kappa-1}}(t_{1..\kappa-1})
=0.$$ Hence, $\prod_{Q_i\in\bQ_\kappa}(L_1^p-Q_i^p)\cdot v_\b=0$.
Similarly, if $b_1\neq 0$, then
$$v_\b\cdot\prod_{Q_j\in\bQ_1}(L_1^p-Q_j^p)=0.$$  Now suppose that
$1\le \alpha<\kappa$, $b_{\alpha}\neq 0$, and set
$L(\alpha)=\prod_{Q_i\in\bQ_\alpha}(L_{1+b_{\alpha+1..\kappa}}^p-Q_i^p)$.
Then, using (\ref{switch}), the observation before Proposition \ref{v-shift} and the fact that any symmetric polynomial
on Murphy operators is central, we deduce that
$L(\alpha)v_\b$ has a factor of the form $$
\prod_{Q_i\in\bQ_\alpha}(L_1^p-Q_i^p) \cdot \biggl(\prod_{k=1}^{t_{1..\alpha-1}}(L_1^p-Q_k^p)\biggr)
\biggl(\prod_{k=t_{1..\alpha}+1}^{t}(L_1^p-Q_k^p)\biggr) .$$
Hence, $L(\alpha)v_\b=0$ by combining the relation
$\prod_{i=1}^t(L_1^p-Q_i^p)=0$ with the last displayed equation. The
second statement
$$v_\b\cdot\prod_{Q_i\in\bQ_\alpha}(L_{b_{1..\alpha-1}+1}^p-Q_i^p)=0,$$
for $\alpha=2,\dots,\kappa$, is equivalent to what we have just
proved because $v_\b
L_{b_{1..\alpha-1}+1}=L_{w_\b^{-1}(b_{1..\alpha-1}+1)}v_\b
           =L_{1+b_{\alpha+1..\kappa}}v_\b$ by
       Proposition~\ref{v-shift}.
\end{proof}

To proceed we recall the cellular basis of the algebras $\HH_{r,n}$,
and the associated combinatorics, introduced in~\cite{DJM:cyc}. A
\textbf{multipartition of $n$} is an ordered $r$--tuple of
partitions $\blam=(\lambda^{(1)},\dots,\lambda^{(r)})$ such that
$|\lambda^{(1)}|+\dots+|\lambda^{(r)}|=n$. Let $\Lambda^+_n$ be the
set of multipartitions of $n$. Then $\Lambda^+_n$ is a poset under
the \textbf{dominance order}, where $\blam\gedom\bmu$ if
$$\sum_{a=1}^{s-1}|\lambda^{(a)}|+\sum_{j=1}^i\lambda^{(s)}_j
\ge\sum_{a=1}^{s-1}|\mu^{(a)}|+\sum_{j=1}^i\mu^{(s)}_j,$$ for all
$1\le s\le r$ and all $i\ge1$.

The \textbf{diagram} of~$\blam$ is the set
$[\blam]
   =\set{(i,j,s)|1\le j\le\lambda^{(s)}_i\text{ for }1\le s\le r}$.
A \textbf{$\blam$-tableau} is a bijection
$\t\map{[\blam]}\{1,2,\dots,n\}$. The $\blam$--tableau $\t$ is
\textbf{standard} if $\t(i,j,s)<\t(i',j',s)$ whenever $i\le i'$, $j\le
j'$ and $(i,j,s)$ and $(i',j',s)$ are distinct elements of $[\blam]$.
Let $\Std(\blam)$ be the set of standard $\blam$--tableaux. Observe
that $\Sym_n$ acts from the right on the set of $\blam$--tableaux. In
particular, if $w\in\Sym_n$ and $\t\in\Std(\blam)$ then $\t w$ is a
$\blam$--tableau, however, it is not necessarily standard.

If $\blam\in\Lambda^+_n$ let
$\Sym_\blam=\Sym_{\lambda^{(1)}}\times\dots\times\Sym_{\lambda^{(r)}}$
be the corresponding Young (or parabolic) subgroup of $\Sym_n$. We set
$$x_\blam=\sum_{w\in\Sym_\blam}T_w\quad\text{and}\quad
  u_\blam^+=\prod_{s=2}^r
    \prod_{k=1}^{|\lambda^{(1)}|+\dots+|\lambda^{(s-1)}|}(L_k-Q'_s).$$
Then $x_\blam$ and $u_\blam^+$ are commuting elements of
$\HH_{r,n}$. Next, if $\s$ is a standard $\blam$-tableau let $d(\s)$
be the corresponding distinguished right coset representative of
$\Sym_\blam$ in $\Sym_n$. Finally, given a pair $(\s,\t)$ of
standard $\blam$--tableaux define $m_{\s\t}=T_{d(\s)}^* x_\blam
u_\blam^+ T_{d(\t)}$, where $*$ is the unique anti--isomorphism of
$\HH_{r,n}$ which fixes $T_0,\dots,T_{n-1}$. Then
$$\set{m_{\s\t}|\s,\t\in\Std(\blam)\text{ for some $\blam\in\Lambda^+_n$}}$$
is a cellular basis of $\HH_{r,n}$ by \cite[Theorem~3.26]{DJM:cyc}.

We can relate $V^\b$ to the combinatorics of the cellular basis $\{m_{\s\t}\}$
by defining $\omega_\b=(\omega_\b^{(1)},\dots,\omega_\b^{(r)})$ to be the
multipartition with
$$\omega_\b^{(s)}=\begin{cases}
    (1^{b_\alpha}),&\text{if }s=pt_{1..\alpha} \text{ for some }\alpha,\\
         (0),&\text{otherwise}.
\end{cases}$$
 From the definitions,
$u_{\omega_\b}^+=\prod_{\alpha=1}^{\kappa-1}\prod_{Q_i\in\bQ_{\alpha+1}}
  (L_1^p-Q_i^p)\dots(L_{b_{1..\alpha}}^p-Q_i^p)$. Hence, using
(\ref{switch}) we obtain the following.

\begin{lem}\label{factorization}
Suppose that $\b\in\Lambda(n,\kappa)$. Then $v_\b=v_\b^-u^+_{\omega_\b}$,
where
$$v_\b^-=\prod_{\alpha=2}^\kappa\prod_{i=1}^{t_{1..\alpha-1}}
  (L_1^p-Q_i^p)\dots(L_{b_\alpha}^p-Q_i^p)\cdot
      T_{w_{b_\alpha,b_{1..\alpha-1}}}$$
and in the product $\alpha$ decreases in order from left to right.
\end{lem}

Following~\cite{DJM:cyc} define
$M^{\omega_\b}=u_{\omega_\b}^+\HH_{r,n}$. By the Lemma, there is a
surjective $\HH_{r,n}$--module homomorphism
$\theta_\b\map{M^{\omega_\b}}V^\b$ given by $\theta_\b(h)=v^-_\b h$,
for all $h\in M^{\omega_\b}$.

Suppose that $\blam$ is a multipartition and that  $\t$ is a
standard $\blam$-tableau. For each integer $k$, with $1\le k\le n$,
define $\comp_\t(k)=s$ if $(i,j,s)$ is the unique node in~$[\blam]$
such that $\t(i,j,s)=k$.

\begin{dfn}
Suppose that $\blam$ is a multipartition of $n$. Define
\begin{align*}
\Std_\b(\blam)&=\set{\t\in\Std(\blam)|%
  \comp_\t(k)\le pt_{1..\alpha}\text{ if }1\le k\le b_{1..\alpha}}\\
\intertext{and}
\Std^+_\b(\blam)&=\set{\t\in\Std(\blam)|%
      pt_{1..\alpha-1}<\comp_\t(k)\le pt_{1..\alpha}\text{ if }
      b_{1..\alpha-1}<k\le b_{1..\alpha}}
\end{align*}
\end{dfn}

Then $\Std_\b(\blam)\ne\emptyset$ if only if
$\sum_{s=1}^{pt_{1..\alpha}}|\lambda^{(s)}|\ge b_{1..\alpha}$ for
$1\le\alpha\le\kappa$ and $\Std^+_\b(\blam)\ne0$ if and only if
$\sum_{s=pt_{1..\alpha-1}+1}^{pt_{1..\alpha}}|\lambda^{(s)}|=b_\alpha$,
for $1\le\alpha\le\kappa$. Hence,
$\Std^+_\b(\blam)\subseteq\Std_\b(\blam)$.

\begin{lem}\label{v-kill}
Suppose that $\b\in\Lambda(n,\kappa)$.
\begin{enumerate}
\item $M^{\omega_\b}$ has basis
$\set{m_{\s\t}|\s\in\Std_\b(\blam), \t\in\Std(\blam)
         \text{ for some $\blam\in\Lambda^+_n$}}$.
\item Suppose that $\s\in\Std_\b(\blam)\setminus\Std^+_\b(\blam)$ and
$\t\in\Std(\blam)$. Then $\theta_\b(m_{\s\t})=0$.
\end{enumerate}
\end{lem}

\begin{proof}
Part (a) is a translation of \cite[Theorem~4.14]{DJM:cyc} into the
current notation. See the proof of \cite[Lemma~3.9]{DM:Morita} for
more details.

For part (b) we follow the proof of \cite[Lemma~3.10]{DM:Morita}.
Let $\c=(c_1,\dots,c_\kappa)$, where
$c_\alpha=|\lambda^{(pt_{1..\alpha-1}+1)}|+\dots+|\lambda^{(pt_{1..\alpha})}|$,
for $1\le\alpha\le\kappa$. Then $c_{1..\alpha}\ge b_{1..\alpha}$,
for $1\le\alpha\le\kappa$, since $\s\in\Std_\b(\blam)$ and $\c\ne\b$
since $\s\not\in\Std^+_\b(\blam)$. Choose $\beta$ to be minimal such
that $c_\beta>b_\beta$. Then $1\le \beta<\kappa$ and if $1\le
\alpha\le\beta$ then $pt_{1..\alpha-1}<\comp_\s(k)\le
pt_{1..\alpha}$ if $b_{1..\alpha-1}<k\le b_{1..\alpha}$ since
$c_1=b_1,\dots,c_{\beta-1}=b_{\beta-1}$ and $\s\in\Std_\b(\blam)$.
Choose $\gamma\geq\beta$ to be minimal such that $b_{\gamma}\neq 0$.
Let $w$ be a permutation of
$\bigl\{b_{1..\gamma}+1,b_{1..\gamma}+2,\dots,n\bigr\}$ of minimal
length such that $\s'=\s w$ is a standard $\blam$--tableau with
$pt_{1..\alpha-1}<\comp_\s(k)\le pt_{1..\alpha}$ if
$c_{1..\alpha-1}<k\le c_{1..\alpha}$ for
$\gamma+1\le\alpha\le\kappa$. (Such a permutation exists because we
can first swap integers $k$, with $c_{1..\kappa-1}<k\le n$ and
$\comp_\s(k)\leq pt_{1..\kappa-1}$, with the integers $l$, where
$b_{1..\kappa-1}<l\le n$ and $pt_{1..\kappa-1}<\comp_\s(l)\leq
pt_{1..\kappa}$; let $\s_1$ be the resulting $\lambda$--tableau.
Then we swap integers $k$, with $c_{1..\kappa-2}<k\le
c_{1..\kappa-1}$ and $\comp_{\s_1}(k)\leq pt_{1..\kappa-2}$, with
the integers $l$, where $b_{1..\kappa-2}<l\le n$ and
$pt_{1..\kappa-2}<\comp_{\s_1}(l)\leq pt_{1..\kappa-1}$; $\cdots$,
and so on; compare \cite[Lemma~3.10]{DM:Morita}.) As a result,
$\comp_\s(k)\le pt_{1..\gamma}$ if $k\le c_{1..\gamma}$. Then
$d(\s)=d(\s')w$, with the lengths adding, so that
$m_{\s\t}=T_w^*m_{\s'\t}$. Furthermore, by construction, there is a
composition $\c'\in\Lambda(n,\kappa)$ such that
$\s'\in\Std_{\c'}(\blam)$, $c'_\alpha=c_\alpha$, for
$1\le\alpha<\beta$ or $\gamma\le\alpha\le\kappa$, and
$c'_{1..\alpha}\ge b_{1..\alpha}$, for $1\le\alpha\le\kappa$. Hence,
$m_{\s'\t}\in M^{\omega_{\c'}}$ by part~(a), so that
$m_{\s'\t}=u_{\omega_{\c'}}^+h$ for some $h\in\HH_{r,n}$. Therefore,
$\theta_\b(m_{\s\t})=v_\b^- T_w^*m_{\s'\t}=v_\b^-
T_w^*v_{\omega_{\c'}}^+h$.

To simplify the notation, for the remainder of the proof set
$w(\alpha)=w_{b_\alpha,b_{1..\alpha-1}}$ and
$u^-(m,s)=\prod_{i=1}^s(L_1^p-Q_i^p)\dots(L_m^p-Q_i^p)$, for
$1\le\alpha\le\kappa$, $1\le m\le n$ and $1\le s\le t$. Similarly,
set $u^+(m,s)=\prod_{i=s}^t(L_1^p-Q_i^p)\dots(L_m^p-Q_i^p)$. Then
$$\theta_\b(m_{\s\t})=v_\b^-T_w^* u_{\omega_{\c'}}^+h
       =\prod_{\alpha=2}^\kappa u^-(b_\alpha,t_{1..\alpha-1})T_{w(\alpha)}
           \cdot T_w^* u_{\omega_{\c'}}^+h,
$$
where the product is taken in order with $\alpha$ decreasing from
left to right. Now, $u^\pm(m,s)$ commutes with $T_i$ if $i\ne m$.
Therefore, since $w$ is a permutation of
$\bigl\{b_{1..\gamma}+1,b_{1..\gamma}+2,\dots,n\bigr\}$, we have
$$\theta_\b(m_{\s\t})
       =\prod_{\alpha=\gamma+1}^\kappa u^-(b_\alpha,t_{1..\alpha-1})T_{w(\alpha)}
           \cdot T_{w}^{\ast}\cdot\prod_{\alpha=2}^{\gamma}
          u^-(b_\alpha,t_{1..\alpha-1}) T_{w(\alpha)}\cdot u_{\omega_{\c'}}^+h,
$$
for some $w'\in\Sym_n$, where again both products are ordered with
$\alpha$ decreasing from left to right. By definition,
$u_{\omega_{\c'}}^+=\prod_{\alpha=1}^{\kappa-1}u^+(c'_{1..\alpha},t_{1..\alpha}+1)$,
where this product can be taken in any order. So,
$$\prod_{\alpha=2}^\gamma u^-(b_\alpha,t_{1..\alpha-1})T_{w(\alpha)}\cdot u_{\omega_{\c'}}^+
  =\prod_{\alpha=2}^\gamma u^-(b_\alpha,t_{1..\alpha-1})T_{w(\alpha)}\cdot
   \prod_{\alpha=1}^{\kappa-1}u^+(c'_{1..\alpha},t_{1..\alpha}+1).
$$
Now,
$w(\alpha)=w_{b_\alpha,b_{1..\alpha-1}}\in\Sym_{b_{1..\alpha}}$. So
if $1\le\alpha<\gamma$ then $T_{w(\alpha)}$ commutes with
$u^+(c'_{1..\gamma},t_{1..\gamma}+1)$ since
$c'_{1..\gamma}=c_{1..\gamma}>b_{1..\alpha}$. Consequently, the last
displayed equation contains
$u^-(b_\gamma,t_{1..\gamma-1})T_{w(\gamma)}u^+(c_{1..\gamma},t_{1..\gamma}+1)$
as a factor. Since $c_{1..\gamma}>b_{1..\gamma-1}$, this element is
equal to
$$v_{b_\gamma,b_{1..\gamma-1}}(t_{1..\gamma-1})
  \prod_{s=t_{1..\gamma}+1}^t(L_{b_{1..\gamma-1}+1}^p-Q_s^p)\dots(L_{c_{1..\gamma}}^p-Q_s^p)=0,$$
where the last equality comes from applying the right hand equation
of Lemma~\ref{v-vanishing} in the special case when $\kappa=2$.
Putting all of these equations together, we have shown that
$\theta_\b(m_{\s\t})=0$, as required.
\end{proof}

Suppose that $\t$ is a standard $\blam$--tableau and that $1\le k\le n$.
Let $\Shape_k(\t)$ be the multipartition with diagram
$\t^{-1}(\{1,\dots,k\})$; that is, $\Shape_k(\t)$ is the
multipartition given by the positions of $\{1,\dots,k\}$ in $\t$. If
$\t\in\Std(\blam)$ and $\v\in\Std(\bmu)$ then we write $\t\gedom\v$ if
$\blam\gdom\bmu$ or if $\blam=\bmu$ and
$\Shape_k(\t)\gedom\Shape_k(\v)$ for $1\le k\le n$. We extend this
partial order to pairs of standard tableaux in the obvious way.

\begin{lem}\label{v-li}
Suppose that $\blam$ is a multipartition of~$n$ and that
$\s\in\Std^+_\b(\blam)$ and $\t\in\Std(\blam)$. Let $\s'=\s
w_\b^{-1}$. Then there exists an invertible element $u\in R$ such
that
$$\theta_\b(m_{\s\t})
 =um_{\s'\t}+\sum_{(\u,\v)\gdom(\s',\t)}r_{\u\v}m_{\u\v},$$
for some $r_{\u\v}\in R$.
\end{lem}

\begin{proof}
By \cite[Prop.~3.7]{JM:cyc-Schaper}, if $1\le k\le n$ and
$p(a-1)<\comp_\t(k)\le pa$ then
$$m_{\s\t}L_k^p=q^{p(j-i)}Q_a^pm_{\s\t}
    +\sum_{(\u,\v)\gdom(\s,\t)}r_{\u\v}m_{\u\v},$$
for some $r_{\u\v}\in R$. Using this formula we can compute
$\theta(m_{\s\t})=v_\b^- m_{\s\t}$ directly, which shows that $m_{\u\v}$
appears with non-zero coefficient in $\theta_\b(m_{\s\t})$ only if
$(\u,\v)\gedom(\s',\t)$. Finally, $m_{\s'\t}$ appears with non-zero
coefficient in $\theta_\b(m_{\s\t})$ because
$\bQ=\bQ_1\vee\dots\vee\bQ_\kappa$ is a partition of $\bQ$ into
$(\eps,q)$-orbits.  (Compare with the proof of
\cite[Lemma~3.11]{DM:Morita}.)
\end{proof}

Suppose that $\blam$ is a multipartition of~$n$. Let
$\HH_{r,n}^\blam$ be the module with basis $\{m_{\u\v}\}$ where $\u$
and $\v$ range over the standard $\bmu$--tableaux with
$\bmu\gdom\blam$. It follows from the general theory of cellular
algebras \cite[Lemma~2.3]{M:Ulect} that $\HH_{r,n}^\blam$ is a
two--sided ideal of $\HH_{r,n}$.

Fix $\s\in\Std(\blam)$. Then, as a vector space, the \textbf{Specht module}
(or cell module) $S(\blam)$ is the module with basis
$\set{m_{\s\t}+\HH_{r,n}^\blam|\t\in\Std(\blam)}$.  The theory of cellular
algebras~\cite[2.4]{M:Ulect} shows that $S(\blam)$ is an $\HH_{r,n}$-module
and that, up to isomorphism, $S(\blam)$ does not depend on the choice
of~$\s$.

Finally, we need the classification of the blocks for $\HH_{r,n}$.
For each $\blam\in\Lambda^+_n$ define a ``content function''
$c_\blam\map R\mathbb N$ by
$$c_\blam(x)=\#\set{(i,j,a+pb)\in[\blam]|%
                0\le a<p\text{ and }x=q^{j-i}\eps^aQ_b},$$
for $x\in R$. Then the Specht modules $S(\blam)$ and $S(\bmu)$ are in
the same block only if $c_\blam(x)=c_\bmu(x)$, for all $x\in R$, by
\cite[Prop.~5.9(ii)]{GL}. (Although we will not
need this we note that the converse is also true by
\cite[Theorem~A]{LM:blocks}.)

With the results that we have now proved we can complete the proof of
Theorem~\ref{DM:Morita} with only minor modifications of the arguments
of~\cite{DM:Morita}. Consequently, we sketch the rest of the proof and
give references to \cite{DM:Morita} for those readers who require more
detail.

\begin{dfn}Suppose that $\b\in\Lambda(n,\kappa)$ and that
    $\s\in\Std^+_\b(\blam)$, $\t\in\Std(\blam)$ for some $\blam\in\Lambda^+_n$.
    \begin{enumerate}
    \item Set $v_{\s\t}=\theta_\b(m_{\s\t})\in V^\b$.
    \item If, in addition, $\t\in\Std^+_\b(\blam)$, then let
      $\theta_{\s\t}\in\End_R(V^\b)$ be the endomorphism
      $\theta_{\s\t}(v_\b h)=v_{\s\t}h$, for all
      $h\in\HH_{r,n}$.
    \end{enumerate}
\end{dfn}

We remark that it is not clear from the definition that the maps
$\theta_{\s\t}$ are well--defined.

For $\b\in\Lambda(n,\kappa)$ let
$$\Lambda_\b^+=\set{\blam\in\Lambda_n^+|
    b_\alpha=|\lambda^{(pt_{1..\alpha-1}+1)}|+\dots+|\lambda^{(pt_{1..\alpha})}|,
    \text{ for }\alpha=1,\dots,\kappa}.$$
Note that $\Std_\b^+(\blam)\ne\emptyset$ if and only if $\blam\in\Lambda_\b^+$.

\begin{prop}\label{V-basis}
Suppose that $\b\in\Lambda(n,\kappa)$. Then:
\begin{enumerate}
    \item $V^\b$ has basis
    $$\set{v_{\s\t}|\s,\in\Std^+_\b(\blam),\t\in\Std(\blam)\text{ for some }
                \blam\in\Lambda^+_\b}.$$
    \item If $\b\ne\c\in\Lambda(n,\kappa)$ then
    $\Hom_{\HH_{r,n}}(V^\b,V^\c)=0$.
    \item $\End_{\HH_{r,n}}(V^\b)$ is a vector space with basis
    $$\set{\theta_{\s\t}|\s,\t\in\Std^+_\b(\blam)\text{ for some }
                  \blam\in\Lambda^+_n}.$$
\end{enumerate}
\end{prop}

\begin{proof}
(a) This follows directly from Lemma~\ref{v-kill}(b) and
Lemma~\ref{v-li}.

(b) As in \cite[Theorem~3.16]{DM:Morita} it follows from part~(a)
and the construction of the Specht modules that $V^\b$ has a
filtration $V^\b=V_1\supset V_2\dots\supset V_k=0$ such that (1)
$V_i/V_{i+1}\cong S(\blam_i)$, for some $\blam_i\in\Lambda^+_n$, and
(2) if $\bmu\in\Lambda^+_n$ then
$$\#\Std^{+}_\b(\bmu)=\#\set{1\le i<k|V_i/V_{i+1}\cong S(\bmu)}.$$
Now, if $\b\ne\c$ and $\blam$ and $\bmu$ are two
multipartitions such that $\Std^{+}_\b(\blam)\ne\emptyset$ and
$\Std^{+}_\c(\bmu)\ne\emptyset$ then it is easy to see (cf.~the
proof of \cite[Cor.~3.17]{DM:Morita}) that $c_\blam\ne\c_\bmu$.
Consequently, by the remarks before the Theorem, the Specht modules
$S(\blam)$ and~$S(\bmu)$ are in different blocks. Therefore, all
of the composition factors of $V^\b$ and~$V^\c$ belong to different
blocks, so $\Hom_{\HH_{r,n}}(V^\b,V^\c)=0$.

(c) The proof is identical to that of
\cite[Theorem~3.19]{DM:Morita}. In outline, the argument is as
follows. By \cite[Theorem~6.16]{DJM:cyc} and Lemma~\ref{v-kill}(a),
$\End_{\HH_{r,n}}(M^{\omega_\b})$ has basis
$\set{\varphi_{\s\t}|\s,\t\in\Std^+_\b(\blam)\text{ for some
}\blam\in\Lambda^+_n}$, where
$\varphi_{\s\t}(u^+_{\omega_\b}h)=m_{\s\t}h$ for all
$h\in\HH_{r,n}$. As in the proof of part~(b), the filtration
$0\subseteq\ker\theta_\b\subset M^{\omega_\b}$ can be extended to a
Specht filtration which is compatible with the Specht filtration of
$V^\b\cong M^{\omega_\b}/\ker\theta_\b$. Using this Specht
filtration and the classification of the blocks of $\HH_{r,n}$ given
above it follows that all of the irreducible constituents of $V_\b$
and $\ker\theta_\b$ belong to different blocks. Therefore, the map
$\theta_\b\map{M^{\omega_\b}}V^\b$ splits.  Let $\theta_\b^{-1}$ be
a \textit{right} inverse to $\theta_\b$. Then a straightforward
calculation shows that
$$\theta_\b\varphi_{\s\t}\theta_\b^{-1}=\begin{cases}
    \theta_{\s\t},&\text{if } \s,\t\in\Std^+_\b(\blam),\\
    0,&\text{otherwise.}
\end{cases}$$
As the maps
$\set{\theta_\b\varphi_{\s\t}\theta_\b^{-1}|\s,\t\in\Std_\b^+(\blam)}$ span
$\End_{\HH_{r,n}}(V^\b)$, this proves part~(c).
\end{proof}

Let
$\HH_\b=\HH_{pt_1,b_1}(\bQ_1)\otimes\dots\otimes%
                \HH_{pt_\kappa,b_\kappa}(\bQ_\kappa)$.
Let $T_i^{(\alpha)}=1\otimes\dots\otimes T_i\otimes\dots\otimes1$ be a
generator of $\HH_\b$, where the $T_i$ occurs in the $\alpha^{\text{th}}$ tensor
factor, for $1\le\alpha\le\kappa$. Then, as an algebra,~$\HH_\b$ is generated
by the elements
$\set{T_i^{(\alpha)}|1\le\alpha\le\kappa\text{ and }0\le i<b_\alpha}$.

We need some combinatorial machinery to describe the $\HH_\b$--modules.
For $\blam\in\Lambda_\b^+$ let
$\blam_\b=(\blam_\b^{(1)},\dots,\blam_\b^{(\kappa)})$, where
$\blam_\b^{(\alpha)}=(\lambda^{(pt_{1..\alpha-1}+1)},\dots,%
          \lambda^{(pt_{1..\alpha})})$.
Then the Specht modules of $\HH_\b$ are all of the form
$S(\blam_\b^{(1)})\otimes\dots\otimes S(\blam_\b^{(\kappa)})$, for
$\blam\in\Lambda_\b^+$, and there is a natural bijection
$\Std^+_\b(\blam) \cong\Std(\blam_\b^{(1)})\times\dots\times
          \Std(\blam_\b^{(\kappa)})$.

Let $\Sym_\b=\Sym_{b_1}\times\dots\times\Sym_{b_\kappa}$, which we
consider as a subgroup of $\Sym_n$ via the natural embedding.
Let $\calD_\b$ be the set of distinguished (minimal length) right coset
representatives for $\Sym_\b$ in $\Sym_n$. Observe that if
$\blam\in\Lambda_\b^+$ then
\begin{equation}\label{b-standard}
\Std(\blam)=\coprod_{d\in\calD_\b}\Std^+_\b(\blam)d.
\end{equation}

Recall that a \textbf{progenerator}, or projective generator, for an
algebra $A$ is a projective $A$--module $V$ which contains every
projective indecomposable $A$--module as a direct summand. The algebras $A$
and $\End_A(V)$ are Morita equivalent and, moreover, every Morita
equivalence arises in this way.

\begin{prop}\label{real Morita}\leavevmode
\begin{enumerate}
\item Let $V=\bigoplus_{\b\in\Lambda(n,\kappa)}V^\b$. Then $V$ is a
progenerator for $\HH_{r,n}$.
\item Suppose that $\b\in\Lambda(n,\kappa)$. Then:
\begin{enumerate}
\item $V^\b$ is a projective $\HH_{r,n}$--module;
\item $\End_{\HH_{r,n}}(V^\b)\cong\HH_\b$; and
\item as left $\HH_\b$--modules, $\HH_\b\cong V^\b_0$ and
    $V^\b=\bigoplus_{d\in\calD_\b}V^\b_0T_d$,
where $V^\b_0$ is the subspace of $V^\b$ with basis
$\set{v_{\s\t}|\s,\t\in\Std^+_\b(\blam)\text{ for some
}\blam\in\Lambda_\b^+}$.
\end{enumerate}
\end{enumerate}
\end{prop}

\begin{proof}[Sketch of proof]
Using Proposition~\ref{V-basis} and (\ref{b-standard}) it is
straightforward to show that
$\HH_{r,n}\cong\bigoplus_{\b\in\Lambda(n,\kappa)}\bigoplus_{d\in\calD_\b}
  T_d^*V^\b$
(see the proof of \cite[Theorem~3.20]{DM:Morita}). Hence, $V^\b$ is
a projective $\HH_{r,n}$--module, proving b(i). Part~(a) now
follows because $V^\b\cong T_d^* V^\b$, for all $d\in\calD_\b$.

Now consider the remaining statements of~(b). By Proposition~\ref{v-shift}
and Lemma~\ref{v-vanishing} there is an action of $\HH_\b$ on $V^\b$ by
left multiplication which is uniquely determined by letting the
generator~$T_i^{(\alpha)}$ of $\HH_\b$ act as left multiplication by
$L_{1+b_{\alpha+1..\kappa}}$ if $i=0$ and by
$T_{i+1+b_{\alpha+1..\kappa}}$ if $1\le i<b_\alpha$. Thus, there is a
map from $\HH_\b$ into $\End_{\HH_{r,n}}(V^\b)$. The argument used
to prove \cite[Theorem~4.7]{DM:Morita} now shows that if
$\blam\in\Lambda_\b^+$ and $\s,\t\in\Std^+_\b(\blam)$ then the map
$\theta_{\s\t}\in\End_{\HH_{r,n}}(V^\b)$ corresponds to left
multiplication by the corresponding Murphy basis element of
$\HH_\b$, where we use the bijection $\Std^{+}_\b(\blam)
\cong\Std(\blam_\b^{(1)})\times\dots\times
          \Std(\blam_\b^{(\kappa)})$.
That is if $h\in\HH_{r,p,n}$ then
$\theta_{\s\t}(v_\b h)=(m_{\s^{(1)}\t^{(1)}}\otimes\dots\otimes
m_{\s^{(\kappa)}\t^{(\kappa)}})v_\b h$, where $\u\in\Std_\b^+(\blam)$ maps to
$(\u^{(1)},\dots,u^{(\kappa)})$ under the bijection above;
see the proof of \cite[Lemma~4.6]{DM:Morita}. This shows that
$\End_{\HH_{r,n}}(V^\b)\cong\HH_\b$. Finally, part b(iii) follows from
b(ii) and the observation that
$V^\b=\bigoplus_{d\in\calD_\b}V^\b_0T_d$, as a left $\HH_\b$--module.
\end{proof}

Notice, in particular, that Theorem~\ref{DM:Morita} is an immediate
Corollary of the Proposition. Later we need the following result
which follows directly from Proposition~\ref{real Morita}b(iii);
compare \cite[Remark 3.15]{DM:Morita}.

\begin{cor}\label{basis}
Suppose that $\b\in\Lambda(n,\kappa)$. Then
$$\set{v_\b L_1^{c_1}\cdots L_{n}^{c_n}T_w|%
w\in\BS_n \text{ and }0\leq c_i<pt_\alpha \text{ whenever }
            b_{1..\alpha-1}< i\le b_{1..\alpha}}$$
is a basis of $v_\b\HH_{r,n,}$.
\end{cor}

We now have the information that we need to start proving
Theorem~\ref{main1} from the introduction.

\section{Morita equivalence theorems for algebras of type $G(r,p,n)$}
In this section we prove Theorem~\ref{main1}, the Morita reduction
theorem for the Hecke algebras of type~$G(r,p,n)$, by analyzing the
structure of $V^\b=v_\b\HH_{r,n}$ as an $\HH_{r,p,n}$--module. We
maintain our notation from the previous section. In particular, we fix
a partitioning $\bQ=\bQ_1\vee\dots\vee\bQ_\kappa$ of $\bQ$ such that
$Q_i$ and~$Q_j$ are in different $(\eps,q)$--orbits whenever
$Q_i\in\bQ_\alpha$, $Q_j\in\bQ_\beta$ and $\alpha\ne\beta$.

Following Ariki~\cite{Ariki:hecke}, for each integer $m$ with $1\leq m\leq n$,
define:
$$S_m=\begin{cases} T_0^{p}, &\text{if $m=1$;}\\
T_0^{-1}L_m, &\text{if $2\leq m\leq n$.}
\end{cases}
$$
The elements $S_1,S_2,\cdots,S_n$ are the {\bf Murphy operators}
of~$\HH_{r,p,n}$. We need these elements to prove the
following fundamental fact.

\begin{lem}\label{Hrpn basis}
    The algebra $\HH_{r,p,n}$ has basis
    $$\set{L_1^{c_1}\dots L_n^{c_n}T_w|w\in\Sym_n, 0\le c_i<r
            \text{ and } c_1+\dots+c_n\equiv0\bmod p}.$$
\end{lem}

\begin{proof}
Ariki~\cite[Prop.~1.6]{Ariki:hecke} showed that $\HH_{r,p,n}$ is the
submodule of $\HH_{r,n}$ with basis
$$
\Big\{S_1^{c_1}\cdots S_{n}^{c_n}T_w\,\Big|\,\begin{matrix}%
    w\in\BS_n, 0\leq c_i<r \text{ for }2\le i\leq n,\\
    \text{and  }0\le pc_1-c_2-\dots-c_n<r
\end{matrix}\Big\}.
$$
Applying the definitions
$S_1^{c_1}\cdots S_{n}^{c_n}T_w=L_1^{pc_1-c_2-\dots-c_n}L_2^{c_2}
       \dots L_n^{c_n}T_w$.
Hence, the Lemma is just a reformulation of Ariki's result.
\end{proof}

Recall from the introduction that there are two algebra
automorphisms $\sigma$ and~$\tau$ of~$\HH_{r,n}$.

\begin{dfn}\label{sigma tau}
The automorphism
$\tau$ is the $K$-algebra automorphism of $\HH_{r,n}$ which is
given by $\tau(h)=T_0^{-1}hT_0$, for all $h\in\HH_{r,n}$.
The map $\sigma$ is the $K$-algebra automorphism of $\HH_{r,n}$ which is
determined by
$$\sigma(T_0)=\eps T_0\quad\And\quad \sigma(T_i)=T_i,
\qquad\text{for } i=1,\dots,n-1.$$
\end{dfn}

The reader can check that $\HH_{r,p,n}$ is the fixed point subalgebra of
$\HH_{r,n}$ under~$\sigma$.

Suppose that $A$ is an algebra with an automorphism $\theta$ of
order~$p$. Define $A\rtimes_\theta\Z_p$ to be the $K$-algebra with
elements $$\set{a\theta^k|a\in A\And0\le k<p}$$ and with
multiplication $a\theta^k\cdot
b\theta^l=a\theta^{k}(b)\theta^{k+l}$, for $a,b\in A$ and $0\le
k,l<p$.  As in the introduction, if $M$ is an $A$--module then we
can define a new $A$--module $M^\theta$ which is isomorphic to $M$
as a vector space but with the $A$--action twisted by $\theta$.
Informally, it is convenient to think of $M^\theta$ as the set of
elements $\set{m\theta|m\in M}$ with $A$-action $m\theta\cdot
a=\big(m\theta(a)\big)\theta$, for $m\in M$ and $a\in A$.

\begin{lem}\label{invariant}Suppose that $0\le b\le n$.  Then
$\sigma(v_\b)=v_\b $ and $\tau(v_\b)\in v_\b \HH_{r,p,n}$.
Consequently, $\bigl(v_\b\HH_{r,n}\bigr)^{\sigma}=v_\b\HH_{r,n}$ and
$\bigl(v_\b\HH_{r,p,n}\bigr)^{\tau}=v_\b\HH_{r,p,n}$ as
$\HH_{r,p,n}$-modules.
\end{lem}

\begin{proof}
Since $\sigma(T_i)=T_i$, for $1\le i<n$, and $\sigma(L_k^p)=L_k^p$,
for $1\le k\le n$, we see that $\sigma(v_\b)=v_\b$.  Furthermore,
$T_0v_\b=v_\b L_{b_{1..\kappa-1}+1}$ and $v_\b
T_0=L_{b_{2..\kappa}+1}v_\b$, by Proposition~\ref{v-shift}, so
$\tau(v_\b)=T_{0}^{-1}v_\b T_{0}=v_\b L_{b_{1..\kappa-1}+1}^{-1}L_1
            =v_\b S_{b_{1..\kappa-1}+1}^{-1}\in v_\b\HH_{r,p,n}$.
 From what we have proved, $\sigma(v_\b\HH_{r,n})=v_\b\HH_{r,n}$.
Consequently, the map $v_\b h\mapsto\sigma(v_\b
h)=\sigma(v_\b)\sigma(h)$ defines a module isomorphism
$v_\b\HH_{r,n}\cong\big(v_\b\HH_{r,n}\big)^\sigma$, for
$h\in\HH_{r,n}$. Similarly, $\big(v_\b\HH_{r,p,n}\big)^\tau\cong
v_\b\HH_{r,p,n}$ as $\HH_{r,p,n}$-modules.
\end{proof}

\begin{prop}\label{basis2}Suppose that $0\le b\le n$. Then
$$\Bigg\{v_\b L_1^{c_1}\cdots L_{n}^{c_n}T_w\,\Bigg|\,\begin{matrix}%
    w\in\Sym_n\text{ and }
    0\le c_i<pt_\alpha\text{ whenever }b_{1..\alpha-1}<i\le b_{1..\alpha}\\
      \text{ and }c_1+\dots+c_n\equiv0\!\!\pmod p
\end{matrix}\Bigg\}
$$
is a basis of $v_\b\HH_{r,p,n}$. In particular,
$\dim v_\b \HH_{r,p,n}=\frac1p\dim v_\b \HH_{r,n}$.
\end{prop}

\begin{proof}
First, observe that $\dim v_\b \HH_{r,p,n}\geq\frac1p\dim v_\b \HH_{r,n}$
since $\HH_{r,n}$ is a free $\HH_{r,p,n}$--module of rank~$p$. By
Corollary~\ref{basis} the number of the elements
given in the statement of the Proposition is exactly
$\frac1p\dim v_\b \HH_{r,n}$. Therefore, it suffices to show that the
elements in the statement in the Proposition span $v_\b \HH_{r,p,n}$.

By Lemma~\ref{Hrpn basis} the module $v_\b \HH_{r,p,n}$ is spanned by the
elements
$$ \Big\{v_\b L_1^{c_1}\cdots L_{n}^{c_n}T_w\,\Big|\,%
  \begin{matrix}w\in\BS_n, 0\leq c_i<r \text{ for } 1\leq i\leq n\\
      c_1+\cdots+c_n\equiv 0\!\!\pmod{p}
  \end{matrix}\Big\}.
$$
The elements $L_1,\dots,L_n$ commute by \cite[Lemma~3.3]{AK}.
Therefore, to prove the Proposition it is enough to show for
$\alpha=1,\dots,\kappa$ that if $b_{1..\alpha-1}<i\le b_{1..\alpha}$
then $v_\b L_i^c$ is a linear combination of terms of the form
$v_\b L_{b_{1..\alpha-1}+1}^{a_{b_{1..\alpha-1}+1}}\dots L_i^{a_{i-1}} T_w$,
where $0\le a_j<pt_\alpha$ for all $j$ and $w$
is an element of the symmetric group on the letters
$\{b_{1..\alpha-1}+1,\dots,i\}$.  We prove this by induction on $i$.

Fix $\alpha$ such that $b_\alpha\ne0$ and  $1\le\alpha\le\kappa$. Suppose first that
$i=i_0$, where $i_0=b_{1..\alpha-1}+1\le b_{1..\alpha}$.
Recall from Lemma~\ref{v-vanishing} that
$$v_\b\cdot\prod_{Q_i\in\bQ_\alpha}(L_{i_0}^p-Q_i^p)=0.$$
Therefore, $v_\b L_{i_0}^{pt_\alpha}$ can be written as a
linear combination of the elements $v_\b L_{i_0}^{pk}$,
for $0\le k<t_\alpha$. Note that, modulo~$p$, we have not changed the
exponent of $L_{i_0}$. Hence, we may assume that $0\le
c_i<pt_\alpha$ when $i=i_0$. Now suppose that
$i_0<i\le b_{1..\alpha}$. Arguing by induction (see
\cite[Lemma~3.3]{AK}), it follows easily that
\begin{equation}\label{rewrite}
L_i^c=q^{-1}T_{i-1}L_{i-1}^cT_{i-1}
       +(1-q^{-1})\sum_{d=1}^{c-1}L_{i-1}^{c-d}L_i^dT_{i-1}.
\end{equation}
Therefore, using Proposition~\ref{switch},
\begin{align*}
    v_\b L_i^c &=q^{-1}v_\b T_{i-1}L_{i-1}^cT_{i-1}
       +(1-q^{-1})\sum_{d=1}^{c-1}v_\b L_{i-1}^{c-d}L_i^dT_{i-1}\\
       &=q^{-1}T_{w_\b^{-1}(i-1)}v_\b L_{i-1}^cT_{i-1}
       +(1-q^{-1})\sum_{d=1}^{c-1}v_\b L_{i-1}^{c-d}L_i^dT_{i-1}.
\end{align*}
If $c\ge pt_\alpha$ then, by induction on $i$, we can rewrite
$v_\b L_{i-1}^c$ as a linear combination of terms of the form
$v_\b L_{i_0}^{a_{i_0}}\dots L_{i-1}^{a_{i-1}} T_w$,
where $0\le a_j<pt_\alpha$ for all
$j$ and $w$ is an element of the symmetric group on the letters
$\{i_0,\dots,i-1\}$. Now, $L_1,\dots,L_n$ commute with each other, and
$T_{i-1}$ commutes with $L_j$ if $j\ne i-1,i$, so
\begin{align*}
    T_{w_\b^{-1}(i-1)}v_\b
L_{i_0}^{a_{i_0}}\dots L_{i-1}^{a_{i-1}} T_w T_{i-1}
   &=v_\b T_{i-1} L_{i_0}^{a_{i_0}}
        \dots L_{i-1}^{a_{i-1}} T_w T_{i-1}\\
   &=v_\b L_{i_0}^{a_{i_0}}
   \dots L_{i-2}^{a_{i-2}}T_{i-1}L_{i-1}^{a_{i-1}}T_w T_{i-1}\\
\end{align*}
Hence, using (\ref{rewrite}) once again, we can rewrite $v_\b L_i^c$ as a linear
combination of terms of the form
$v_\b L_{i_0}^{a_{i_0}}\dots L_i^{a_i} T_w$,
where $0\le a_j<pt_\alpha$ for all
$j$ and $w$ is an element of the symmetric group on the letters
$\{i_0,\dots,i\}$. This proves our claim. Moreover, this completes the
proof of the Proposition because, modulo~$p$, the sums of the exponents
of $L_1,\dots,L_n$ are unchanged in all of the formulae above.
\end{proof}

If $M$ is an $\HH_{r,p,n}$--module let
$M\ind_{\HH_{r,p,n}}^{\HH_{r,n}}=M\otimes_{\HH_{r,p,n}}\HH_{r,n}$ be
the corresponding \textbf{induced} $\HH_{r,n}$--module. Similarly,
if~$N$ is an $\HH_{r,n}$--module let
$N\res_{\HH_{r,p,n}}^{\HH_{r,n}}$ be the \textbf{restriction} of $N$
to $\HH_{r,p,n}$. Since $\HH_{r,n}$ is free as an
$\HH_{r,p,n}$--module both induction and restriction are exact
functors.

\begin{cor}\label{4lm}
\begin{enumerate}
  \item $v_\b \HH_{r,n}\cong\bigl(v_\b \HH_{r,p,n}\bigr)\ind_{\HH_{r,p,n}}^{\HH_{r,n}}$,

  \item $\bigl(v_\b \HH_{r,n}\bigr)\res^{\HH_{r,n}}_{\HH_{r,p,n}}
       \cong\big(v_\b\HH_{r,p,n}\big)^{\oplus p},$

  \item $\bigl(v_\b \HH_{r,n}\bigr)\res^{\HH_{r,n}}_{\HH_{r,p,n}}\!\uparrow^{\HH_{r,n}}_{\HH_{r,p,n}}
      \cong\big(v_\b \HH_{r,n}\big)^{\oplus p}$.
\end{enumerate}
\end{cor}

\begin{proof}
Since $\HH_{r,n}=\bigoplus_{k=0}^{p-1}T_0^k\HH_{r,p,n}$, there is a surjective
homomorphism from $(v_\b \HH_{r,p,n}\bigr)\ind_{\HH_{r,p,n}}^{\HH_{r,n}}$
onto $v_\b\HH_{r,n}$. By Corollary~\ref{basis} and Proposition~\ref{basis2}
both modules have the same dimension so this map must be an
isomorphism, proving~(a). Part~(b) now follows from
Proposition~\ref{real Morita}b(iii); alternatively, use part~(a) and
Lemma~\ref{invariant}. Part~(c) follows from parts~(a) and~(b).
\end{proof}

\begin{cor} \label{dim} Suppose that $0\le b\le n$ as above. Then
$$
\frac1p\dim\End_{\HH_{r,p,n}}\!\big(v_\b \HH_{r,n}\big)
     =\dim\End_{\HH_{r,n}}\!\big(v_\b \HH_{r,n}\big)
     =p\dim\End_{\HH_{r,p,n}}\!\big(v_\b \HH_{r,p,n}\big).
$$
\end{cor}

\begin{proof} The left and right hand equalities follow using
Corollary~\ref{4lm} and Frobenius reciprocity.
\end{proof}

We can now prove Theorem~\ref{main1} from the introduction.

\begin{thm}
Suppose that $\bQ=\bQ_1\vee\dots\vee\bQ_{\kappa}$, where
$Q_i\in\bQ_{\alpha}$ and $Q_j\in\bQ_{\beta}$ are in the same
$(\eps,q)$--orbit only if $\alpha=\beta$. Then $\HH_{r,p,n}$ is
Morita equivalent to the algebra
$$\bigoplus_{\b\in\Lambda(n,\kappa)}\HH_\b\rtimes\Z_p$$
\end{thm}

\begin{proof} By Proposition~\ref{real Morita},
$\bigoplus_{\b\in\Lambda(n,\kappa)}V^\b$ is a
progenerator for $\HH_{r,n}$. Hence, by restriction, it is also a progenerator
for $\HH_{r,p,n}$. By Proposition~\ref{V-basis}(b), Frobenius reciprocity
and Corollary~\ref{4lm}(c) if $\b\ne\c$ then
$\Hom_{\HH_{r,p,n}}(V^\b,V^\c)=0$. Therefore, $\HH_{r,p,n}$ is Morita
equivalent to
$$\End_{\HH_{r,p,n}}\Big(\bigoplus_{\b\in\Lambda(n,\kappa)}V^\b\Big)
    =\bigoplus_{\b\in\Lambda(n,\kappa)}\End_{\HH_{r,p,n}}\big(V^\b\big).$$
Hence, to prove the Proposition it suffices to show
that
$$
\End_{\HH_{r,p,n}}\!\big(v_\b \HH_{r,n}\big)\cong\HH_\b\rtimes\Z_p.
$$
for all $\b\in\Lambda(n,\kappa)$.

Recall that each of the algebras
$\HH_{pt_\alpha,b_\alpha}(\bQ_{\alpha})$, for $1\leq
\alpha\leq\kappa$, has an automorphism $\sigma_\alpha$ of order~$p$.
The automorphism $\sigma_1\otimes\dots\otimes\sigma_\kappa$ acts
diagonally on the algebra~$\HH_{\b}$. Note that
$\langle\sigma_1\otimes\dots\otimes\sigma_\kappa\rangle\cong\Z_p$.
By Proposition~\ref{real Morita}, we have that $$
\HH_{\b}\cong\End_{\HH_{r,n}}\!\big(v_\b \HH_{r,n}\big)
\hookrightarrow \End_{\HH_{r,p,n}}\!\big(v_\b \HH_{r,n}\big).$$ On
the other hand, $\sigma(v_{\b})=v_{\b}$ by Lemma~\ref{invariant}.
So, $\sigma$ induces an automorphism of $v_\b\HH_{r,p,n}$ which is given
by $v_\b h\mapsto\sigma(v_\b h)=v_\b\sigma(h)$, for all
$h\in\HH_{r,p,n}$. Hence, we have an injective map
$ \Z_p\hookrightarrow\End_{\HH_{r,p,n}}\!\big(v_\b
\HH_{r,n}\big).
$
Since $\sigma$ is an outer automorphism of $\HH_{r,p,n}$ it follows
that we have an embedding
$$
\HH_\b\rtimes\Z_p\hookrightarrow\End_{\HH_{r,p,n}}\!\big(v_\b
\HH_{r,n}\big).
$$
Using Corollary~\ref{dim} to compare the dimensions on both
sides of this equation, we conclude that
$
\End_{\HH_{r,p,n}}\!\big(v_\b \HH_{r,n}\big)\cong\HH_\b\rtimes\Z_p.
$
\end{proof}

If instead of $V^\b=v_\b\HH_{r,n}$ we consider the
$\HH_{r,p,n}$--module $v_\b\HH_{r,p,n}$ then we obtain a second Morita
reduction theorem for $\HH_{r,p,n}$. To state this result, if
$\b\in\Lambda(n,\kappa)$ set
$\HH_{p,\b}=\HH_{pt_1,p,b_1}(\bQ_1)\otimes\dots\otimes
           \HH_{pt_\kappa,p,b_\kappa}(\bQ_\kappa).$
Observe that $\HH_{p,\b}$ is a subalgebra of $\HH_\b$ and that
$\dim\HH_\b=p^\kappa\dim\HH_{p,\b}$. Next, let $\HH_{p,\b}'$ be
the subalgebra of $\HH_\b$ generated by $\HH_{p,\b}$ and the elements
$\set{T_0^{(1)}\big(T_0^{(\alpha)}\big)^{-1}|1<\alpha\le\kappa}$.

\begin{prop}\label{main1 two}
Suppose that $\bQ=\bQ_1\vee\cdots\vee\bQ_{\kappa}$, where $Q_i$ and
$Q_j$ are in different $(\eps,q)$--orbits whenever
$Q_i\in\bQ_{\alpha}$ and $Q_j\in\bQ_{\beta}$ for some
$\alpha\neq\beta$. Let $\b\in\Lambda(n,\kappa)$. Then $\HH_{r,p,n}$
is Morita equivalent to the algebra
$$\bigoplus_{\b\in\Lambda(n,\kappa)}\HH_{p,\b}'.$$
\end{prop}

\begin{proof}
By Proposition~\ref{real Morita}(a), $\bigoplus_\b v_\b\HH_{r,n}$ is a
progenerator for $\HH_{r,n}$.  Therefore, $\bigoplus_\b
v_\b\HH_{r,p,n}$ is a progenerator for~$\HH_{r,p,n}$ by
Corollary~\ref{4lm}(b). Furthermore, if $\b\ne\c\in\Lambda(n,\kappa)$
then $\Hom_{\HH_{r,n}}(v_\b\HH_{r,n},v_\c\HH_{r,n})=0$ by
Proposition~\ref{V-basis}(b).  By Corollary~\ref{4lm} and Frobenius
reciprocity, $\Hom_{\HH_{r,p,n}}\big(v_\b\HH_{r,p,n},v_\c\HH_{r,p,n}\big)=0$.
Combining these results we see that $\HH_{r,p,n}$ and
$$
\End_{\HH_{r,p,n}}\big(\bigoplus_{\b\in\Lambda(n,\kappa)}v_\b\HH_{r,p,n}\big)
 =\bigoplus_{\b\in\Lambda(n,\kappa)}\End_{\HH_{r,p,n}}\big(v_\b\HH_{r,p,n}\big)
$$
are Morita equivalent.

Fix $\b\in\Lambda(n,\kappa)$ and let
$E_\b=\End_{\HH_{r,p,n}}(v_\b\HH_{r,p,n})$. To complete the proof it
is enough to show that $E_\b\cong\HH_{p,\b}'$. As a left
$\HH_\b$--module, $v_\b\HH_{r,n}$ is isomorphic to a direct sum of
$[\Sym_n{:}\Sym_\b]$ copies of the regular representation of~$\HH_\b$
by Corollary~\ref{basis}. By Proposition~\ref{real Morita}b(iii),
$\HH_{p,\b}$ acts faithfully on $v_\b\HH_{r,p,n}$ by restriction and this
action commutes with the action of $\HH_{r,p,n}$ from the right.  Hence,
we can identify $\HH_{p,\b}$ with a subalgebra of $E_\b$.

By Lemma~\ref{invariant} and Proposition~\ref{real Morita}b(iii),
$\tau(v_\b)=L_1^{-1}L_{b_{2..\kappa}+1}v_\b\in v_\b\HH_{r,p,n}$ acts on
$v_\b\HH_{r,p,n}$
in the same way that $T_0^{(1)}\big(T_0^{(\kappa)}\big)^{-1}
    =T_0\otimes1\otimes\dots\otimes 1\otimes T_0^{-1}\in\HH_\b$ acts on
$v_\b\HH_{r,p,n}$. More generally, for
$\alpha=1,\dots,\kappa-1$ let $\rho_\alpha$ be the automorphism of
$v_\b\HH_{r,p,n}$ given by left multiplication by
$L_{b_{2..\kappa}+1}L_{b_{\alpha+1..\kappa}+1}^{-1}\in\HH_\b$. By
Proposition~\ref{v-shift},
$$L_{b_{2..\kappa}+1}L_{b_{\alpha+1..\kappa}+1}^{-1}v_\b
        =v_\b L_1L_{b_{1..\alpha-1}+1}^{-1}
    =v_\b S_{b_{1..\alpha-1}+1}^{-1}\in v_\b\HH_{r,p,n}.$$
Therefore, $\rho_\alpha\in E_\b$, for $2\le\alpha\le\kappa$, since $\rho_\alpha$
commutes with the action of $\HH_{r,p,n}$.
Thus, $\rho_\alpha$ coincides with the action of
$T_0^{(1)}\big(T_0^{(\alpha)}\big)^{-1}\in\HH_\b$ on $V^\b$ in
Proposition~\ref{real Morita}. Consequently,
$\rho_\alpha\notin\HH_{p,\b}$ and the automorphisms
$\rho_2,\dots,\rho_\kappa$ commute.  By
definition, $\HH_{p,\b}'$ is isomorphic to the subalgebra of $E_\b$
generated by $\HH_{p,\b}$ and $\rho_2,\dots,\rho_\kappa$. Hence,
$\HH_{p,\b}'$ is a subalgebra of $E_\b$.

For each $\alpha$, the map $\rho_\alpha^p$ acts as left multiplication by
$\big(T_0^{(1)}\big)^p\big(T_0^{(\alpha)}\big)^{-p}\in\HH_\b$, so
$\rho_\alpha^p\in\HH_{p,\b}$. Note that the extensions of the endomorphisms
$\rho_\alpha,\dots,\rho_\alpha^{p-1}$ to $\End_{\HH_{r,n}}(v_\b\HH_{r,n})$,
for $2\le\alpha\le\kappa$, are all linearly independent since
$\End_{\HH_{r,n}}(V^\b)\cong\HH_b$ by Proposition~\ref{real Morita}. As
these maps act on different components of $\HH_{p,\b}$ it follows that
$\dim\HH_{p,\b}'=p^{\kappa-1}\dim\HH_{p,\b}$.
However, by Lemma~\ref{dim} and Proposition~\ref{real Morita}b(ii),
$$\dim E_\b=\frac 1p\dim\End_{\HH_{r,n}}(V^\b)
          =\frac1p\dim\HH_\b
      =p^{\kappa-1}\dim\HH_{p,\b}.$$
Therefore, $\dim\HH_{p,\b}'=p^{\kappa-1}\dim\HH_{p,\b}=\dim E_\b$. So
$E_\b\cong\HH_{p,\b}'$, as required.
\end{proof}

\section{Splittable decomposition numbers} In this section we use
Clifford theory to show that if an algebra can be written as a
semidirect product then its decomposition numbers are determined by
the corresponding``$p'$-splittable'' decomposition numbers of a
related family of algebras. First, we recall some general results
about the representation theory of semidirect product algebras. The
basic references for this topic are
\cite{C&R,Macdonald:wreath,RamRam}.

Let $A$ be a finite dimensional algebra over an algebraically closed
field $K$ and suppose that $\theta$ is an algebra automorphism of~$A$ of
order~$p$. We identify $\Z_p$ with the group generated by $\theta$ and
consider the algebra $A\rtimes\Z_p$. Then, as a set,
$$A\rtimes\Z_p=\set{a\theta^k|a\in A\And 0\le k\le p}$$ and the
multiplication in $A\rtimes\Z_p$ is defined by
$$
(a\theta^k)\cdot(b\theta^m)=a\theta^{k}(b)\cdot\theta^{k+m},
$$
for $a,b\in A$ and $0\le k,m<p$. If $H$ is a subgroup of $\Z_p$ we
identify $A\rtimes H$ with a subalgebra of $A\rtimes\Z_p$ in the
natural way. In particular, by taking $H=1$ we can view $A$ as a
subalgebra of $A\rtimes\Z_p$. Moreover, there are natural induction
and restriction functors between the module categories of all of
these algebras.

Suppose that $L$ is an $A$--module.  Then we can twist $L$ by $\theta$
to get a new $A$-module~$L^\theta$.  As a vector space we
set $L^{\theta}=L$, and we define the action of $A$ on $L^{\theta}$ by
$$
v\cdot a:=v\theta(a),\quad\text{for all}\,\, v\in L \text{ and }
a\in A.
$$
It is straightforward to check that $L$ is irreducible if and only if
$L^\theta$ is irreducible.

The {\bf inertia group} of $L$ is the group
$$G_{L}:=\set{\theta^k\in\Z_p|L\cong L^{\theta^k}}.$$ Then $G_L$ is
a subgroup of $\Z_p$ and, in particular, it is cyclic.  Let
$l=|G_{L}|$.  Then $p=lk$ and $G_{L}$ is generated by $\theta^k$.
Recall that we have fixed a primitive $p$th root of unity $\eps\in K$.
Since $K$ is algebraically closed we can choose an $A$-module isomorphism
$\phi\map L{L^{\theta^k}}$ such that
$\phi^l=1_L$, the identity map on $L$. For each integer~$i\in\Z$
define the $(A\rtimes G_{L})$-module $L_{l,i}$ as follows: as a
vector space $L_{l,i}:=L$ and the action of $(A\rtimes G_{L})$ on
$L_{l,i}$ is given by: $$
v\cdot\bigl(a\theta^{mk}\bigr):=\eps^{mki}\phi^{m}(va),\quad
\text{for all } m\in\Z, v\in L_{l,i}\text{ and }a\in A.  $$
It is easy to check that $L_{l,i}$ is an $(A\rtimes G_{L})$-module and, by
definition, that $L_{l,i+l}\cong L_{l,i}$, for all $i\in\Z$.

Recall that $\Irr(A)$ is the complete set of isomorphism classes of
simple $A$-modules.  Let $L\in\Irr(A)$. Then, since $L_{l,i}\downarrow_{A}\cong L$,
it follows that $L_{l,i}$ is a simple $(A\rtimes G_{L})$-module. In
fact,
it is shown in \cite{C&R,Macdonald:wreath,RamRam} that
$$
\Bigl\{L_{l,1}\ind_{A\rtimes G_{L}}^{A\rtimes\Z_p},\cdots,
   L_{l,l}\ind_{A\rtimes
   G_{L}}^{A\rtimes\Z_p}\Bigm|L\in\Irr(A)/{\sim} \text{ and }
   l=|G_{L}|\Bigr\},
$$
where $L'\sim L$ if and only if $L'\cong L^{\theta^i}$ for some
integer $i$, is a complete set of pairwise non--isomorphic simple
$(A\rtimes\Z_p)$-modules.

\begin{lem}\label{2lm}
Suppose that $L$ is a simple $A$--module and that $H$ is a subgroup
of~$G_L$. Let $l=|G_L|$, $h=|H|$ and fix $i$ with $1\le i\le l$.
Then $L_{l,i}\res^{A\rtimes G_L}_{A\rtimes H}$ is an irreducible
$(A\rtimes H)$-module and
$$L_{l,i}\res^{A\rtimes G_L}_{A\rtimes H}\ind_{A\rtimes H}^{A\rtimes G_L}
             \cong\bigoplus_{j=1}^{[G_L:H]}L_{l,i+hj}.$$
\end{lem}

\begin{proof}
As remarked above, we can view $A$ and $A\rtimes H$ as subalgebras
of $A\rtimes G_L$. Since $L_{l,i}\res^{A\rtimes G_L}_A\cong L$ is
irreducible we see that $L_{l,i}\res^{A\rtimes G_L}_{A\rtimes H}$ is
irreducible. Next, observe that $\theta^{p/h}$ is a generator of
$H$. Hence, by the argument above if $1\le j\le l=|G_L|$
then $L_{l,i+hj}\res^{A\rtimes G_L}_{A\rtimes H}
    \cong L_{l,i}\res^{A\rtimes G_L}_{A\rtimes H}$
is irreducible. Therefore, by Frobenius reciprocity and Schur's Lemma,
\begin{align*}
    \Hom_{A\rtimes G_L}\!\big(L_{l,i+hj},
L_{l,i}\res^{A\rtimes G_L}_{A\rtimes H}\!\uparrow^{A\rtimes
G_L}_{A\rtimes H}\!\big) &\cong \Hom_{A\rtimes
H}\!\big(L_{l,i+hj}\!\downarrow^{A\rtimes G_L}_{A\rtimes H},
L_{l,i}\res^{A\rtimes G_L}_{A\rtimes H}\big)\\
  &\cong K.
\end{align*}
The lemma now follows by comparing dimensions.
\end{proof}

Now suppose that we have a modular system $(F,\O,K)$ such that
$A=A_K$ has an $\O$-lattice $A_{\O}$ which is an $\O$-algebra,
$\theta$ can be lifted to an automorphism of~$A_{\O}$ of order~$p$,
$F$ is an algebraically closed field of characteristic zero, and
$A_{F}:=A_{\O}\otimes_{\O}F$ is a (split) semisimple $F$-algebra. We
abuse notation and write $\theta$ for the corresponding automorphism
of $A_F$. Note that if $H$ is a subgroup of $\Z_p$ then $A_K\rtimes
H\cong(A_\O\rtimes H)\otimes_\O K$, so that we also have a modular
system for the algebras $A_F\rtimes H$ and $A_K\rtimes H$.

By definition, $\Irr(A_F)$ is the complete set of isomorphism
classes of simple $A_{F}$-modules --- the ``semisimple''
$A_F$--modules. As the automorphism $\theta$ lifts to $A_F$ for each
simple $A_F$--module $S\in\Irr(A_F)$ we have an inertia group
$G_S\le\Z_p$ and, as above, we can define $(A_F\rtimes G_S)$-modules
$S_{s,j}$, for $j\in\Z$ where $s=|G_S|$. Consequently,
$$ \Bigl\{S_{s,1}\ind_{A_F\rtimes G_{S}}^{A_F\rtimes\Z_p},\cdots,
   S_{s,s}\ind_{A_F\rtimes
   G_{S}}^{A_F\rtimes\Z_p}\Bigm|S\in\Irr(A_F)/{\sim}
   \text{ and } s=|G_{S}|\Bigr\}
$$
is a complete set of pairwise non-isomorphic simple
$(A_{F}\rtimes\Z_p)$-modules.

Suppose that $S\in\Irr(A_F)$ and that $D\in\Irr(A_K)$ and let
$s=|G_S|$ and  $d=|G_D|$. Given $i$ and $j$ with $1\le i\le s$ and
$1\le j\le d$, we want to determine the decomposition number
$$[S_{s,i}\ind_{A_{F}\rtimes G_{S}}^{A_{F}\rtimes\Z_p}:
              D_{d,j}\ind_{A\rtimes G_{D}}^{A\rtimes\Z_p}],$$
which gives the multiplicity of $D_{d,j}\ind_{A\rtimes
G_{D}}^{A\rtimes\Z_p}$ as an irreducible composition factor of a
modular reduction of $S_{s,i}\ind_{A\rtimes G_{S}}^{A\rtimes\Z_p}$.

\begin{dfn}
Suppose that $S\in\Irr(A_F)$ and $D\in\Irr(A_K)$ and set $s=|G_S|$ and
$d=|G_D|$. Then the pair $(S,D)$ has \textbf{cyclic decomposition
numbers} if
$$
[S_{s,i}\!\uparrow^{A_F\rtimes\Z_p}_{A_F\rtimes G_S}:%
               D_{d,j}\!\uparrow^{A\rtimes\Z_p}_{A\rtimes G_D}]
  =[S_{s,i+1}\!\uparrow^{A_F\rtimes\Z_p}_{A_F\rtimes G_S}:%
               D_{d,j+1}\!\uparrow^{A\rtimes\Z_p}_{A\rtimes G_D}],
$$
for all $i,j\in\Z$.
\end{dfn}

In the next section we show that all pairs of irreducible
$\HH_{r,n}$-modules have cyclic decomposition numbers.

\begin{prop} \label{condtm}
Let $S\in\Irr(A_F)$ and $D\in\Irr(A_K)$ and suppose that $(S,D)$ has
cyclic decomposition numbers. Set $s=|G_S|$, $d=|G_D|$ and let
$d_0=\gcd(s,d)$. Then
$$
[S_{s,i+d_0l}\!\uparrow^{A_F\rtimes\Z_p}_{A_F\rtimes G_S}:%
           D_{d,j+d_0l'}\!\uparrow^{A\rtimes\Z_p}_{A\rtimes G_D}]
    =[S_{s,i}\!\uparrow^{A_F\rtimes\Z_p}_{A_F\rtimes G_S}:%
           D_{d,j}\!\uparrow^{A\rtimes\Z_p}_{A\rtimes G_D}]
$$
for all $i,j,l,l'\in\Z$.
\end{prop}

\begin{proof}
Since the  groups $G_S$, $G_D$ and $G_0=G_S\cap G_D$ are all cyclic subgroups
of~$\Z_p$, we have that $|G_0|=\gcd(|G_S|,|G_D|)=\gcd(s,d)=d_0$, so that
there exist integers $u$ and $v$ such that $d_0=us+vd$. Suppose
that $l,l'\in\Z$. Recall that $S_{s,i+ms}\cong S_{s,i}$ and
$D_{d,j+md}\cong D_{d,j}$, for all $m\in\Z$. Then, using the assumption that
$(S,D)$ has cyclic decomposition numbers for the third equality, we find that
\begin{align*}
[S_{s,i+d_0l}\!\uparrow^{A_F\rtimes\Z_p}_{A_F\rtimes G_S}:%
          D_{d,j+d_0l'}\!\uparrow^{A\rtimes\Z_p}_{A\rtimes G_D}]
    &= [S_{s,i+(us+vd)l}\!\uparrow^{A_F\rtimes\Z_p}_{A_F\rtimes G_S}:%
          D_{d,j+(us+vd)l'}\!\uparrow^{A\rtimes\Z_p}_{A\rtimes G_D}]\\
    &= [S_{s,i+vdl}\!\uparrow^{A_F\rtimes\Z_p}_{A_F\rtimes G_S}:%
          D_{d,j+usl'}\!\uparrow^{A\rtimes\Z_p}_{A\rtimes G_D}]\\
    &= [S_{s,i-usl'}\!\uparrow^{A_F\rtimes\Z_p}_{A_F\rtimes G_S}:%
          D_{d,j-vdl}\!\uparrow^{A\rtimes\Z_p}_{A\rtimes G_D}]\\
    &= [S_{s,i}\!\uparrow^{A_F\rtimes\Z_p}_{A_F\rtimes G_S}:%
          D_{d,j}\!\uparrow^{A\rtimes\Z_p}_{A\rtimes G_D}],
\end{align*}
as required.
\end{proof}

\begin{dfn}\label{split3}
Suppose that $S\in\Irr(A_F)$ and $D\in\Irr(A_K)$ and let $s=|G_S|$
and $d=|G_D|$ as above. Then the decomposition number
$[S_{s,i}\ind_{A_{F}\rtimes G_{S}}^{A_{F}\rtimes\Z_p}:%
         D_{d,j}\ind_{A\rtimes G_{D}}^{A\rtimes\Z_p}]$
is \textbf{$p$-splittable} if $G_S=\Z_p=G_D$.
\end{dfn}

We will see in Corollary~\ref{H-splittable} below that this definition of
$p$--splittable decomposition number agrees with Definition~\ref{split}
when applied to the algebras~$\HH_{r,p,n}$.

\begin{cor}\label{splittable}
Suppose that $S\in\Irr(A_F)$ and $D\in\Irr(A_K)$ have cyclic decomposition numbers
and let $G_0=G_S\cap G_D$, $s=|G_S|$ and $d=|G_D|$. Then
$$[S_{s,i}\ind_{A_F\rtimes G_{S}}^{A_F\rtimes\Z_p}:%
            D_{d,j}\ind_{A\rtimes G_{D}}^{A\rtimes\Z_p}]
         =\sum_{1\leq a\leq p/d}[S_{s,i}\res^{A_F\rtimes G_S}_{A_F\rtimes G_0}:
            D^{\theta^a}_{d,j}\res^{A\rtimes G_D}_{A\rtimes G_0}]
$$
for all $1\le i\le s$ and all $1\le j\le d$. In particular, all of the
decomposition numbers appearing in the summation are $p'$-splittable
in the sense of Definition~\ref{split3}, where $p'=|G_0|$.
\end{cor}

\begin{proof}
Observe that if $L\not\sim L'$ are simple $(A\rtimes G_0)$-modules
then the induced modules $L\ind_{A\rtimes G_0}^{A\rtimes\Z_p}$ and
$L'\ind_{A\rtimes G_0}^{A\rtimes\Z_p}$ have no common irreducible
composition factors. Next observe that if $1\le i\le s$ then
$S_{s,i}\res^{A_F\rtimes G_S}_{A_F\rtimes G_0}$ is an irreducible
$(A\rtimes G_0)$-module by Lemma~\ref{2lm}. Similarly, if $1\le j\le
d$ then $D_{d,j}\res^{A\rtimes G_S}_{A\rtimes G_0}$ is an
irreducible $(A\rtimes G_0)$-module.  The first claim now follows by
combining these observations with Lemma \ref{2lm} and
Proposition~\ref{condtm}. Finally, all of these decomposition numbers are
$p'$-splittable because the inertia groups of the modules $S_{s,i}$
and $D_{d,j}$ inside $G_0$ are both equal to~$G_0$.
\end{proof}

Corollary~\ref{splittable} reduces the calculation of decomposition numbers
of the algebra $A\rtimes\Z_p$ to determining the $p'$-splittable
decomposition numbers of the subalgebras $A\rtimes\Z_{p'}$, where $p'$
divides~$p$.

We now apply Corollary~\ref{splittable} in the special case where
$A$ has a tensor product decomposition. That is, we suppose that
$A=A^{(1)}\otimes\dots\otimes A^{(\kappa)}$, for some finite
dimensional $K$--algebras $A^{(1)},\dots,A^{(\kappa)}$ which are
equipped with automorphisms $\theta_1,\dots,\theta_\kappa$,
respectively, of order~$p$. Set
$\theta=\theta_1\otimes\dots\otimes\theta_\kappa$. We assume that
the tensor product decomposition of $A$ is compatible with the
modular system $(F,\O,K)$ so that $A_F=A_F^{(1)}\otimes\dots\otimes
A_F^{(\kappa)}$.

Fix an integer $f\in\{1,\dots,\kappa\}$ and let
$S^{(f)}\in\Irr(A^{(f)}_F)$ and $D^{(f)}\in\Irr(A_K^{(f)})$ be
irreducible modules. Then $S=S^{(1)}\otimes\dots\otimes
S^{(\kappa)}\in\Irr(A_F)$ and $D=D^{(1)}\otimes\dots\otimes
D^{(\kappa)}\in\Irr(A_K)$. Further, $G_S=G_{S^{(1)}}\cap\dots\cap
G_{S^{(\kappa)}}$, and $G_D=G_{D^{(1)}}\cap\dots\cap
G_{D^{(\kappa)}}$.  Set $s=|G_S|$ and $d=|G_D|$. Then
$s=|G_S|=\gcd(|G_{S^{(1)}}|,\dots,|G_{S^{(\kappa)}}|)$ and
$d=|G_D|=\gcd(|G_{D^{(1)}}|,\dots,|G_{D^{(\kappa)}}|)$.

Suppose that $(S,D)$ has cyclic decomposition numbers and let
$G_0=G_S\cap G_D$. Then, by Corollary~\ref{splittable}, the
decomposition numbers of $A\rtimes\Z_p$ are completely determined by
the decomposition numbers of the form
$[S_{s,i}\res^{A_F\rtimes G_S}_{A_F\rtimes G_0}:%
  D^{\theta^a}_{d,j}\res^{A_K\rtimes G_D}_{A\rtimes G_0}]$,
for $1\leq a\leq p/|G_D|$,$1\le i\le s$ and $1\le j\le d$.

\begin{thm} \label{main3} Suppose that
$S=S^{(1)}\otimes\dots\otimes S^{(\kappa)}\in\Irr(A_F)$ and
$D=D^{(1)}\otimes\dots\otimes D^{(\kappa)}\in\Irr(A_K)$. Let
$s=|G_S|$, $d=|G_D|$ and $G_0=G_S\cap G_D$ and set $d_0=|G_0|$.
Then
$$ [S_{s,i}\res^{A_F\rtimes G_S}_{A_F\rtimes G_0}:%
    D_{d,j}\res^{A_K\rtimes G_D}_{A_K\rtimes G_0}]
    =\sum_{\substack{0\le j_1,\dots,j_\kappa<d_0\\%
                     j_1+\dots+j_\kappa\equiv (\kappa-1)i+j\pmod{d_0}}}
             \prod_{\alpha=1}^{\kappa}
    [S^{(\alpha)}_{d_0,i}:D^{(\alpha)}_{d_0,j_\alpha}]%
$$
for all $i$ and $j$ with $1\le i\le s$ and $1\le j\le d$.
\end{thm}

\begin{proof}
Suppose that $R\in\{F,K\}$. Let $k=\frac p{d_0}=|\Z_p/G_0|$, so that
$G_0=\langle\theta^k\rangle$. Consider the algebra
$$\widehat A_R=RG_0\otimes(A_R^{(1)}\rtimes\langle\theta_1^k\rangle)\otimes\dots\otimes%
                    (A_R^{(\kappa)}\rtimes\langle\theta_{\kappa}^k\rangle).$$
Then it is straightforward to check
that there is an embedding of algebras $A_R\rtimes
G_0\hookrightarrow\widehat A_R$ given by
$$
(a_1\otimes\cdots\otimes a_{\kappa})\theta^{mk }\mapsto \theta^{mk
}\otimes \bigl(a_1\theta_1^{mk
}\bigr)\otimes\cdots\otimes\bigl(a_{\kappa}\theta_{\kappa}^{mk
}\bigr),
$$
for $a_1,\cdots,a_{\kappa}\in A_{R}$ and $m\in\Z$.

Let $L=L^{(1)}\otimes\dots\otimes L^{(\kappa)}$ be a simple $A_R$-module,
where $L=S$ if $R=F$, or $L=D$ if $R=K$. Let $l=|G_L|$ and suppose that
$1\le i\le l$. Then $L_{l,i}$ is a simple $(A_R\rtimes G_L)$-module. Recall
that the modules $L_{l,i}$ are defined using a fixed isomorphism
$\phi\map L L^{\theta^{p/l}}$ satisfying $\phi^l=1_L$. We may assume that $\phi$ is
compatible with the tensor decomposition of $L$; that is,
$\phi=\widetilde{\phi}_1\otimes\dots\otimes\widetilde{\phi}_\kappa$, where for
$f=1,\dots,\kappa$, the order of the inertia group $G_{L^{(f)}}$ of $L^{(f)}$ in
$\Z_p$ is $l_{f}$,
$\widetilde{\phi}_f=(\phi_f)^{l_{f}/l}$,
$\phi_f\map{L^{(f)}}{\big(L^{(f)}\big)^{\theta_f^{p/l_f}}}$ is an
isomorphism defining the module $L_{l_f,i}^{(f)}$. Note that $$
L_{l_f,i}^{(f)}\downarrow^{A_R^{(f)}\rtimes G_{L^{(f)}}}_{A_R^{(f)}\rtimes\langle
\theta_f^k\rangle}\cong L_{d_0,i}^{(f)}.
$$
Applying the definitions, given
any integers $i,i_0,i_1,\dots,i_\kappa$ with $i\equiv
i_1+\dots+i_\kappa-i_0\pmod{d_0}$ there is a natural isomorphism of
$(A_R\rtimes G_0)$-modules
$$L_{l,i}\res^{A_R\rtimes G_L}_{A_R\rtimes G_0}\cong
\big(\eps^{-i_0}\otimes L^{(1)}_{d_0,i_1}\otimes\dots%
   \otimes L^{(\kappa)}_{d_0,i_\kappa}\big)%
    \res^{\widehat A_R}_{A_R\rtimes G_0},\leqno(\dag)$$
where $\eps^{-i_0}$ is the one
dimensional representation of $G_{0}=\langle\theta^k\rangle$ upon
which $\theta^k$ acts as multiplication by $\eps^{-i_0}$.

By $(\dag)$, if $1\le i\le s$ then $S_{s,i}\res^{A_F\rtimes
G_S}_{A_F\rtimes G_0}\cong
\big(\eps^{-(\kappa-1)i}\otimes S^{(1)}_{d_0,i}\otimes\dots\otimes S^{(\kappa)}_{d_0,i}\big)%
    \res^{\widehat A_F}_{A_F\rtimes G_0}$.
Therefore, we can find the $(A_F\rtimes G_0)$-module composition
factors of $S_{s,i}\res^{A_F\rtimes G_S}_{A_F\rtimes G_0}$ by first
finding the $\widehat A_K$-module composition factors of
$\eps^{-(\kappa-1)i}\otimes S^{(1)}_{d_0,i}\otimes\dots\otimes
S^{(\kappa)}_{d_0,i}$ and then restricting to $A_{K}\rtimes G_0$.
The $\widehat A_K$-module composition factors of
$\eps^{-(\kappa-1)i}\otimes S^{(1)}_{d_0,i}\otimes\dots\otimes
S^{(\kappa)}_{d_0,i}$ are all of the form
$\eps^{-(\kappa-1)i}\otimes D^{(1)}_{d_0,j_1}\otimes\dots\otimes%
        D^{(\kappa)}_{d_0,j_\kappa}$, where
$0\le j_1,\dots,j_\kappa<d_0$. Further, by $(\dag)$,
if $j\equiv j_1+\dots+j_\kappa-(\kappa-1)i\pmod{d_0}$ then
$$D_{d,j}\res^{A_K\rtimes G_D}_{A_K\rtimes G_0}\cong
\big(\eps^{-(\kappa-1)i}\otimes D^{(1)}_{d_0,j_1}\otimes\dots\otimes
        D^{(\kappa)}_{d_0,j_\kappa}\big)%
    \res^{\widehat A_K}_{A_K\rtimes G_0}$$
The theorem now follows.
\end{proof}

\begin{cor}\label{orbitreduction}
Suppose that $S=S^{(1)}\otimes\dots\otimes S^{(\kappa)}\in\Irr(A_F)$ and
$D=D^{(1)}\otimes\dots\otimes D^{(\kappa)}\in\Irr(A_K)$ have cyclic
decomposition numbers. Let $s=|G_S|$, $d=|G_D|$ and $G_0=G_S\cap G_D$ and
set $d_0=|G_0|$. Then
$$[S_{s,i}\ind_{A_F\rtimes G_{S}}^{A_F\rtimes\Z_p}:%
            D_{d,j}\ind_{A\rtimes G_{D}}^{A\rtimes\Z_p}]
         =\sum_{\substack{0\le j_1,\dots,j_\kappa<d_0\\%
                     j_1+\dots+j_\kappa\equiv (\kappa-1)i+j\pmod{d_0}}}
             \prod_{\alpha=1}^{\kappa}
    [S^{(\alpha)}_{d_0,i}:D^{(\alpha)}_{d_0,j_\alpha}]%
$$
for all $i$ and $j$ with $1\le i\le s$ and $1\le j\le d$.
\end{cor}

\begin{proof} This follows from Corollary~\ref{splittable} and Theorem~\ref{main3}.
\end{proof}

\section{Reduction theorems for the decomposition numbers of $\HH_{r,p,n}$}

We now combine the results of last two sections to show that the
decomposition numbers of the cyclotomic Hecke algebras $\HH_{r,p,n}$ are
completely determined by the $p'$-splittable decomposition numbers of an
explicitly determined family of cyclotomic Hecke algebras. We first show
that the simple $\HH_{r,n}$-modules always have cyclic decomposition
numbers.

\begin{lem}\label{simple Morita}
The algebras $\H_{r,p,n}$ and $\HH_{r,n}\rtimes\Z_p$ are Morita
equivalent.
\end{lem}

\begin{proof}As a right $\H_{r,p,n}$--module
$\HH_{r,n}=\bigoplus_{k=0}^{p-1}T_0^k\HH_{r,p,n}$ by
Lemma~\ref{Hrpn basis}. Consequently,
$\H_{r,n}$ is a progenerator for $\HH_{r,p,n}$, so $\H_{r,p,n}$ is
Morita equivalent to $\End_{\HH_{r,p,n}}(\H_{r,n})$. Observe that
$\sigma\in\End_{\HH_{r,p,n}}(\HH_{r,n})$ since $\sigma$ is trivial on
$\HH_{r,p,n}$. Furthermore, as vector spaces,
\begin{align*}
\End_{\HH_{r,p,n}}(\H_{r,n})
&\cong\Hom_{\HH_{r,p,n}}(\HH_{r,n}\res_{\HH_{r,p,n}}^{\HH_{r,n}},
                     \HH_{r,p,n}^{\oplus p})\\
 &\cong\Hom_{\HH_{r,p,n}}(\HH_{r,n}\res_{\HH_{r,p,n}}^{\HH_{r,n}},
        \HH_{r,p,n})^{\oplus p}\\
 &\cong\Hom_{\HH_{r,n}}(\HH_{r,n},\HH_{r,p,n}\ind_{\HH_{r,p,n}}^{\HH_{r,n}})^{\oplus p}\\
 &\cong\Hom_{\HH_{r,n}}(\HH_{r,n},\HH_{r,n})^{\oplus p}
 \cong\HH_{r,n}^{\oplus p}.
 \end{align*}
where the third isomorphism comes from Frobenius reciprocity. Hence,
by counting dimensions,
$\End_{\HH_{r,p,n}}(\H_{r,n})\cong\H_{r,n}\rtimes\Z_p$.
(Alternatively, apply Theorem~A with $\kappa=1$.)
\end{proof}

The proof of this Lemma gives a Morita equivalence from the
category of (finite dimensional right)
$\HH_{r,n}^R\rtimes\Z_p$--modules to the category of
$\HH_{r,p,n}^R$--modules. This functor sends the finite dimensional
$(\HH_{r,n}\rtimes\Z_p)$--module $M$ to the $\HH_{r,p,n}$--module
$$F(M)=M\otimes_{\HH_{r,n}^R\rtimes\Z_p}\HH_{r,n}^R.$$

\begin{lem}\label{restriction}
Suppose that $L$ is a simple $\HH_{r,n}^R$--module. Then
$$L\res^{\HH_{r,n}^R}_{\HH_{r,p,n}^R}
   \cong F\big(L\ind^{\HH_{r,n}^R\rtimes\Z_p}_{\HH_{r,n}^R}\big).$$
\end{lem}

\begin{proof}
Applying the definitions and standard properties of tensor products,
$$F\big(L\ind^{\HH_{r,n}^R\rtimes\Z_p}_{\HH_{r,n}^R}\big)
   = \big(L\otimes_{\HH_{r,n}^R}\HH_{r,n}^R\rtimes\Z_p\big)
       \otimes_{\HH_{r,n}^R\rtimes\Z_p}\HH_{r,n}^R
   \cong L=L\res^{\HH_{r,n}^R}_{\HH_{r,p,n}^R}.$$
\end{proof}

\begin{prop}[\protect{\cite[(2.2)]{Genet:graded}, \cite[(2.2)]{GenetJacon} and
    \cite[(5.4), (5.5), (5.6)]{Hu:ModGppn}}] \label{cyclic}
Suppose that $L$ is a simple $\HH^R_{r,n}$-module and let $k>0$ be
minimal such that $L\cong L^{\sigma^{k}}$. Then $1\leq k\leq p$ and
$l:=\frac pk$ is the smallest positive integer such that
$L'\cong{L'}^{\tau^{l}}$ whenever~$L'$ is a simple
$\HH_{r,p,n}$-submodule of~$L$.

Now fix an isomorphism $\phi\map L L^{\sigma^{k}}$ such that
$\phi^l=1$ and for $i\in\Z$ define
$$
L_i:=\bigl\{v\in L\bigm|\phi(v)=\eps^{-ik}v\bigr\}.
$$
Then $L\res^{\HH^R_{r,n}}_{\HH^R_{r,p,n}}=L_0\oplus\dots\oplus L_{l-1}$.
Moreover, $L_i=L_{i+l}$ and
$L_{i+1}\cong L_i^{\tau}$, for any $i\in\Z$. Consequently,
$ L\res^{\HH_{r,n}^R}_{\HH^R_{r,p,n}}
  \cong L_0\oplus {L_0}^{\tau}\oplus\cdots\oplus{L_0}^{\tau^{l-1}}.
$
\end{prop}

For each simple $\HH_{r,n}^R$-module $L$ we henceforth fix an
isomorphism $\phi\map L L^{\sigma^{k}}$ such that $\phi^l=1$, where
$k$ and $l=\frac pk$ as in the Lemma. Observe that $l=|G_L|$, where
$G_L$ is the inertia group of $L$. For each integer $i$ we have
defined $\HH_{r,p,n}$--modules $L_i$ and $L_{l,i}$. The next result gives the
connection between these two modules.

\begin{lem} \label{satis}
Suppose that $L$ is a simple $\HH_{r,n}^R$-module with inertia group
$G_L$ and let $l=|G_L|$. Then, for each $i\in\Z$, we have
$$
L_i\cong
F\bigl(L_{l,i}\ind^{\HH^R_{r,n}\rtimes\Z_p}_{\HH^R_{r,n}\rtimes
G_L}\bigr).
$$
\end{lem}

\begin{proof}
As in the proof of Lemma~\ref{restriction},
$F\bigl(L_{l,i}\ind^{\HH^R_{r,n}\rtimes\Z_p}_{\HH^R_{r,n}\rtimes
G_L}\bigr)
    \cong L_{l,i}\otimes_{\HH^R_{r,n}\rtimes G_L}\HH_{r,n}^R$.
Therefore, there is a natural map $\psi\map
{F\bigl(L_{l,i}\ind^{\HH^R_{r,n}\rtimes\Z_p}_{\HH^R_{r,n}\rtimes
G_L}\bigr)}L$ given by $\psi(v\otimes_{\HH_{r,n}\rtimes\Z_d} h)=vh$,
for $v\in L_{l,i}$ and $h\in\HH_{r,n}^R$. Clearly, $\psi\ne0$ so it
suffices to show that the image of $\psi$ is contained in $L_i$.
Now, if $v\in L_{l,i}$ and $h\in\HH_{r,n}^R$ as above then, using
the definition of $L_{l,i}$, we see that
$$ \eps^{ik}\phi(vh)=v\cdot (h\sigma^k)
           =\psi\bigl(v\cdot (h\sigma^k)\otimes_{\HH_{r,n}\rtimes\Z_d} 1\bigr)
           =\psi\bigl(v\otimes_{\HH_{r,n}\rtimes\Z_d}h\bigr)
           =vh.
$$
Hence, $\psi(vh)\in L_i$ as required.
\end{proof}

\begin{cor} \label{verifycyclic}
Suppose that $S\in\Irr(\HH_{r,n})$ and $D\in\Irr(\HH_{r,n}^K)$ and
let $s=|G_S|$ and $d=|G_D|$. Then
$$
[S_{s,i}\!\uparrow^{\HH_{r,n}\rtimes\Z_p}_{\HH_{r,n}\rtimes G_S}:%
               D_{d,j}\!\uparrow^{\HH_{r,n}\rtimes\Z_p}_{\HH_{r,n}\rtimes G_D}]
  =[S_{s,i+1}\!\uparrow^{\HH_{r,n}\rtimes\Z_p}_{\HH_{r,n}\rtimes G_S}:%
               D_{d,j+1}\!\uparrow^{\HH_{r,n}\rtimes\Z_p}_{\HH_{r,n}\rtimes G_D}],
$$
for all $i,j\in\Z$. That is, $(S,D)$ has cyclic decomposition numbers.
\end{cor}

\begin{proof}Using Lemma~\ref{satis} and Proposition~\ref{cyclic} we
have
\begin{align*}
[S_{s,i+1}\!\uparrow^{\HH_{r,n}\rtimes\Z_p}_{\HH_{r,n}\rtimes G_S}:%
               D_{d,j+1}\!\uparrow^{\HH_{r,n}\rtimes\Z_p}_{\HH_{r,n}\rtimes G_D}]
  &=[S_{i+1}:D_{j+1}]
   =[S_i^\tau:D_j^\tau]
   =[S_i:D_j]\\
  &=[S_{s,i}\!\uparrow^{\HH_{r,n}\rtimes\Z_p}_{\HH_{r,n}\rtimes G_S}:%
               D_{d,j}\!\uparrow^{\HH_{r,n}\rtimes\Z_p}_{\HH_{r,n}\rtimes G_D}],
\end{align*}
as required.
\end{proof}

Note that Proposition~\ref{cyclic} and Lemma~\ref{satis} also imply that
the two notions of $p$--splittable decomposition numbers for the algebras
$\HH_{r,p,n}$ coincide.

\begin{cor}\label{H-splittable}
    Suppose that $S\in\Irr(\HH_{r,n}^F)$ and $D\in\Irr(\HH_{r,n}^K)$ and
    let $s=|G_S|$ and $d=|G_D|$. Fix integers $i,j\in\Z$. Then the
    decomposition number
    $$[S_{s,i}\ind_{\HH_{r,n}^F\rtimes G_{S}}^{\HH_{r,n}^F\rtimes\Z_p}:%
       D_{d,j}\ind_{\HH_{r,n}^K\rtimes G_{D}}^{\HH_{r,n}^K\rtimes\Z_p}]$$
   is $p$--splittable in the sense of Definition~$\ref{split3}$ if and only
   if the decomposition number
   $$[F(S_{s,i}\ind_{\HH_{r,n}^F\rtimes G_{S}}^{\HH_{r,n}^F\rtimes\Z_p}):%
       F(D_{d,j}\ind_{\HH_{r,n}^K\rtimes G_{D}}^{\HH_{r,n}^K\rtimes\Z_p})]$$
       is $p$--splittable in the sense of Definition~$\ref{split}$.
\end{cor}

\begin{thm}\label{lastthm}
Then every decomposition number of
$\HH_{r,p,n}(\bQ)$ is equal to a sum of $p'$--splittable
decomposition number of some cyclotomic Hecke algebra
$\HH_{r,p',n}(q,{\bQ}')$, where $p=kp'$ and
$$\bQ'=(Q_1,\eps Q_1,\dots,\eps^{k-1}Q_1,Q_2,\dots,\eps^{k-1}Q_2,\dots,
 Q_t,\dots,\eps^{k-1}Q_t).$$
\end{thm}

\begin{proof}
By Proposition~\ref{cyclic} every irreducible $\HH_{r,p,n}^F$-module
is equal to $S_i$ for some $S\in\Irr(\HH_{r,n}^F)$ and with $1\le i\le s=|G_S|$.
Similarly, every irreducible $\HH_{r,p,n}^K$-module is equal to $D_j$ for some
$D\in\Irr(\HH_{r,n}^K)$ and with $1\le j\le s=|G_D|$. Therefore, by
Lemma \ref{satis}, Corollary \ref{verifycyclic} and Corollary \ref{splittable},
$$[S_i:D_j]=\sum_{1\leq a\leq\frac pd}[S_{s,i}\res^{\HH_{r,n}^F\rtimes G_S}_{\HH_{r,n}^F\rtimes G_0}:%
            D^{\theta^a}_{d,j}\res^{\HH_{r,n}^F\rtimes G_D}_{\HH_{r,n}^K\rtimes G_0}],
$$
where $G_0=G_S\cap G_D$. Suppose that $G_0=\langle\sigma^k\rangle$
and write $p=kp'$.  Then we have shown that $[S_i:D_j]$ is a sum of
$p'$--splittable decomposition number of $\HH_{r,n}\rtimes\Z_{p'}$.

As at the beginning of section~2, write $r=pt$. Then $r=p'kt$ and in
$\HH_{r,p,n}(\bQ)$ the `order relation' for $T_0$ is
$$ 0=\prod_{b=1}^{t}\big(T_0^p-Q_b^p\big)
    =\prod_{b=1}^{t}\prod_{a=0}^{k-1}\big(T_0^{p'}-(\eps^{a}Q_b)^{p'}\big).$$
Observe that the right hand side is the `order relation' for
$T_0^{p'}$ in $\HH_{r,p',n}(\bQ')$. It now follows using
Lemma~\ref{simple Morita} that $\HH_{r,n}\rtimes\Z_{p'}$ is Morita
equivalent to $\HH_{r,p',n}(q,{\bQ}')$, where the parameters $\bQ'$
are as given in the statement of the theorem. This completes the
proof of the theorem.
\end{proof}
\medskip

\begin{proof}[\bf Proof of Theorem~\ref{main2}] This follows from an
recursive application of Theorem~\ref{main1},
Corollary~\ref{orbitreduction}, Corollary \ref{verifycyclic} and Theorem~\ref{lastthm}.
\end{proof}

Theorem~\ref{main2} gives a recursive algorithm for computing all of
the decomposition numbers of a cyclotomic Hecke algebra
$\HH_{r,p,n}(\bQ)$ in terms of the $p'$-splittable decomposition
numbers of a family of ``smaller'' cyclotomic Hecke algebras
$\HH_{r',p',n'}(\bQ')$, where $1\le r'\leq r$, $1\le n'\leq n$, $1\le
p'\mid p$, and the parameters $\bQ'$ are contained in a single
$(\varepsilon',q)$-orbit of $\bQ$, where $\varepsilon'$ is a primitive
$p'$th root of unity. Therefore, the $p'$-splittable
decomposition numbers of the cyclotomic Hecke algebras of type
$G(r',p',n')$ completely determine the decomposition numbers of all
cyclotomic Hecke algebras $\HH_{r,p,n}(\bQ)$.

\section*{Acknowledgments}

\thanks{The first author was supported partly by the National Natural
Science Foundation of China (Project 10771014) and the Program NCET.  He
wishes to thank the School of Mathematics and Statistics at University of
Sydney for their hospitality during his visit in 2006.  Both authors were
supported, in part, by an Australian Research Council discovery grant.}


\begin{thebibliography}{10}

\bibitem{Ariki:hecke}
{\sc S.~Ariki}, {\em Representation theory of a {Hecke} algebra of
  {$G(r,p,n)$}}, J.~Algebra, {\bf 177} (1995), 164--185.

\bibitem{AK}
{\sc S.~Ariki and K.~Koike}, {\em A {H}ecke algebra of {$({\bf {Z}}/r{\bf
  {Z}})\wr{\mathfrak {S}}\sb n$} and construction of its irreducible
  representations}, Adv. Math., {\bf 106} (1994), 216--243.

\bibitem{BM:cyc}
{\sc M.~Brou\'e and G.~Malle}, {\em {Zyklotomische Heckealgebren}}, Asterisque,
  {\bf 212} (1993), 119--189.

\bibitem{C&R}
{\sc C.~W. Curtis and I.~Reiner}, {\em Representation theory of finite groups
  and associative algebras}, Interscience, 1962.

\bibitem{DJM:cyc}
{\sc R.~Dipper, G.~James, and A.~Mathas}, {\em Cyclotomic $q$--{Schur}
  algebras}, Math.~Z., {\bf 229} (1999), 385--416.

\bibitem{DM:Morita}
{\sc R.~Dipper and A.~Mathas}, {\em Morita equivalences of {Ariki--Koike}
  algebras}, Math. Zeit., {\bf 240} (2002), 579--610.

\bibitem{Genet:graded}
{\sc G.~Genet}, {\em On decomposition matrices for graded algebras}, J.
  Algebra, {\bf 274} (2004), 523--542.

\bibitem{GL}
{\sc J.~J. Graham and G.~I. Lehrer}, {\em Cellular algebras}, Invent. Math.,
  {\bf 123} (1996), 1--34.

\bibitem{Hu:ModGppn}
{\sc J.~Hu}, {\em Modular representations of {H}ecke algebras of type
  {$G(p,p,n)$}}, J. Algebra, {\bf 274} (2004), 446--490.


\bibitem{Hu:CrystalGppn}
\leavevmode\vrule height 2pt depth -1.6pt width 23pt, {\em Crystal bases and
  simple modules for {H}ecke algebra of type {$G(p,p,n)$}}, Representation
  Theory, {\bf 11} (2007), 16--44.

\bibitem{Hu:simpleGrpn}
\leavevmode\vrule height 2pt depth -1.6pt width 23pt, {\em The number of simple
  modules for the {H}ecke algebras of type {$G(r,p,n)$} (with an appendix by Xiaoyi Cui)}, J. Algebra, to appear.


\bibitem{Hu:Grpn}
\leavevmode\vrule height 2pt depth -1.6pt width 23pt, {\em The
representation theory of the cyclotomic {H}ecke algebras of type
{$G(r,p,n)$}}, in: Geometry, Analysis and Topology of Discrete Groups, Advanced Lectures in Mathematics, Higher Education Press, (2008), 196--230.


\bibitem{JM:cyc-Schaper}
{\sc G.~D. James and A.~Mathas}, {\em The {Jantzen} sum formula for cyclotomic
  $q$--{Schur} algebras}, Trans. Amer. Math. Soc., {\bf 352} (2000),
  5381--5404.

\bibitem{LM:blocks}
{\sc S.~Lyle and A.~Mathas}, {\em Blocks of cyclotomic {H}ecke algebras},
  Adv. Math., {\bf 216} (2007), 854--878.

\bibitem{Macdonald:wreath}
{\sc I.~G. Macdonald}, {\em Polynomial functors and wreath products}, J. Pure
  Appl. Algebra, {\bf 18} (1980), 173--204.

\bibitem{M:Ulect}
{\sc A.~Mathas}, {\em {Hecke algebras and Schur algebras of the symmetric
  group}}, Univ. Lecture Notes, {\bf 15}, Amer. Math. Soc., 1999.

\bibitem{m:cyclosurv}
\leavevmode\vrule height 2pt depth -1.6pt width 23pt, {\em The representation
  theory of the {A}riki-{K}oike and cyclotomic {$q$}-{S}chur algebras}, in
  Representation theory of algebraic groups and quantum groups, Adv. Stud. Pure
  Math., {\bf 40}, Math. Soc. Japan, Tokyo, 2004, 261--320.

\bibitem{Malle:fake}
{\sc G.~Malle}, {\em On the rationality and fake
degrees of characters of cyclotomic algebras}, J.
Math. Sci. Univ. Tokyo, {\bf 6} (1999), 647--677.

\bibitem{RamRam}
{\sc A.~Ram and J.~Ramagge}, {\em Affine {H}ecke algebras, cyclotomic {H}ecke
  algebras and {C}lifford theory}, in A tribute to C. S. Seshadri (Chennai,
  2002), Trends Math., Birkh\"auser, Basel, 2003, 428--466.

\bibitem{GenetJacon}
{\sc G.~Rassemusse-Genet and N.~Jacon}, {\em Modular representations of
  cyclotomic {H}ecke algebras of type {$G(r,p,n)$}}, Internat. Math.
Res. Notices, (2006), 1--18.

\end{thebibliography}

\end{document}